\documentclass[11pt, reqno]{amsart}

\usepackage{amsfonts}
\usepackage{amssymb}
\usepackage{amsthm}
\usepackage{amsmath}
\usepackage{enumerate}

\usepackage[style = nature, sorting = nyt]{biblatex}
\newcommand{\mres}{\mathbin{\vrule height 1.6ex depth 0pt width
0.13ex\vrule height 0.13ex depth 0pt width 1.3ex}}

\usepackage{graphicx}

\newtheorem{thm}{Theorem}[section]

\newtheorem{lemma}[thm]{Lemma}

\newtheorem{rmk}[thm]{Remark}

\newtheorem{prop}[thm]{Proposition}

\numberwithin{equation}{section}
\begin{document}
\title{Self-shrinkers whose asymptotic cones fatten }
\author{Daniel Ketover}\address{Rutgers University\\  Busch Campus - Hill Center \\ 110 Freylinghausen Road, Piscataway NJ 08854 USA}
\thanks{The author was partially supported by NSF DMS-1401996.}
 \email{dk927@math.rutgers.edu}
\maketitle

\begin{abstract}
For each positive integer $g$ we use variational methods to construct a genus $g$ self-shrinker $\Sigma_g$ in $\mathbb{R}^3$ with entropy less than $2$ and prismatic symmetry group $\mathbb{D}_{g+1}\times\mathbb{Z}_2$.   For $g$ sufficiently large,  the self-shrinker $\Sigma_g$ has two graphical asymptotically conical ends and the sequence $\Sigma_g$ converges on compact subsets to a plane with multiplicity two as $g\to\infty$. Angenent-Chopp-Ilmanen conjectured the existence of such self-shrinkers in 1995 based on numerical experiments.  Using these surfaces as initial conditions for large $g$,  we obtain examples of mean curvature flows in $\mathbb{R}^3$ with smooth initial non-compact data that evolve non-uniquely after their first singular time.

\end{abstract}

\section{Introduction}

A hypersurface $\Sigma\subset\mathbb{R}^3$ is called a \emph{self-shrinker} if 
\begin{equation}
H_\Sigma=\frac{\langle x, \mathbf{n}\rangle}{2},
\end{equation}
where $H_\Sigma$ denotes the mean curvature of $\Sigma$, and $\mathbf{n}_\Sigma$ is a choice of unit normal and $x$ denotes the position vector. 

If $\Sigma$ is a self-shrinker,  then it is the $t=-1$ time-slice of the ancient mean curvature flow (MCF) evolving by homothety
\begin{equation}
\Sigma_t = \sqrt{-t}\Sigma\mbox{ for } t\leq 0.
\end{equation}

By Huisken's monotonicity formula \cite{Huisken} and an argument of White and Ilmanen,  self-shrinkers are realized as blowups at singularities of the MCF.    The sphere of radius two $\mathbb{S}^2_*$ centered about the origin is a self-shrinker,  as well as all cylinders $\mathbb{S}^1_*\times\mathbb{R}$ of radius $\sqrt{2}$ centered about an axis through the origin.  Angenent discovered an embedded rotationally symmetric self-shrinking torus \cite{A}.

For $y\in\mathbb{R}^3$,  and $\tau>0$,  consider the functional

\begin{equation}
F_{y,\tau}(\Sigma) =\frac{1}{4\pi\tau} \int_\Sigma e^{\frac{-|x-y|^2}{4\tau}} d\mu_x.
\end{equation}

The surface $\Sigma$ is a self-shrinker if and only if $\Sigma$ is a critical point for the functional $F:=F_{0,1}$,  which we call the ``Gaussian area."   In other words,  self-shrinkers are precisely the minimal surfaces in the Gaussian metric $(\mathbb{R}^3,\frac{1}{4\pi} e^{-|x|^2/4}\delta_{ij})$.   The Gaussian metric is incomplete and has Ricci curvature approaching $-\infty$ as $|x|\rightarrow\infty$.   On the other hand,  the metric satisfies a Frankel-type property \cite{F} in that any two self-shrinkers intersect.    In this sense,  there is an analogy between minimal surfaces in $\mathbb{S}^3$ and self-shrinkers (cf. \cite{KZ}).

Colding-Minicozzi \cite{CM2} introduced the entropy functional,  which measures the complexity of $\Sigma$ at all scales:
\begin{equation}
\lambda(\Sigma) = \sup_{y\in\mathbb{R}^3, \tau\in\mathbb{R}^+} F_{y,\tau}(\Sigma) = \sup_{y\in\mathbb{R}^3, \tau\in\mathbb{R}^+} F_{0,1}(\tau(\Sigma-y)).
\end{equation}

Entropy is non-increasing along the MCF,  and when $\Sigma$ is a self-shrinker,  the entropy is realized
\begin{equation}
\lambda(\Sigma) = F_{0,1}(\Sigma)=F(\Sigma).
\end{equation}

The normalization is chosen so that $\lambda(\mathbb{R}_*^2) = 1$.  Stone \cite{S} computed that $\lambda(\mathbb{S}^2_*)=4/e\approx 1.47$ and $\lambda(\mathbb{S}^1_*\times\mathbb{R})=\sqrt{2\pi/e}\approx 1.52$.   Bernstein-Wang \cite{BW} showed that any self-shrinker has entropy at least that of $\mathbb{S}^2_*$ and that the cylinder $\mathbb{S}^1_*\times\mathbb{R}$ has the third lowest entropy among self-shrinkers.  

So far, the only non-compact self-shrinkers aside from the cylinder and plane have been obtained by Kapouleas, Kleene and M\o ller \cite{KKM} and independently Nguyen (\cite{Ng1}, \cite{Ng2}, \cite{Ng3}).  Their family has antiprismatic symmetry and consists of high genus surfaces with one conical end.  The surfaces resemble a desingularization of the self-shrinking sphere and plane in the limit that the genus tends to infinity.   Later this family was extended by Buzano-Nguyen-Schulz for all genera using a variational method\footnote{It has yet to be established that the two families coincide.} \cite{BNS}.   

In this paper,  we construct new examples of non-compact self-shrinkers in $\mathbb{R}^3$.  The family was predicted by numerical experiments of Angenent-Chopp-Ilmanen \cite{ACI} (see also \cite{I}) in 1995.  In his ICM lecture in 2002, White sketched work with Ilmanen toward constructing such self-shrinkers \cite{W}.

Let $H$ denote the $xy$-plane $\{z=0\}$ and for any $0<r<\infty$ let us set 
\begin{equation}
C_r = \{(x,y,0)\subset\mathbb{R}^3\;|\; x^2+y^2=r^2\}\subset H.
\end{equation}

Recall that an \emph{asymptotically conical} surface $\Sigma$ embedded in $\mathbb{R}^3$ is one for which $\lim_{\rho\to 0^+}\rho\Sigma = C$ where $C$ is a regular cone in $\mathbb{R}^3$ and the convergence is in $C^\infty_{loc}(\mathbb{R}^3\setminus (0,0,0))$.  In this case, the \emph{link} of $\Sigma$ is defined to be $C\cap\mathbb{S}^2$.  

The following is our main result:

\begin{thm}[Self-shrinking doubled plane]\label{main}
For each integer $g>0$ there exists an embedded self-shrinker $\Sigma_g\subset\mathbb{R}^3$ so that the following hold:
\begin{enumerate}[(a)]
\item $\lambda(\Sigma_g)<2$ and $\lim_{g\rightarrow\infty} \lambda(\Sigma_g)=2$. \label{1}
\item $\Sigma_g$ is invariant under the prismatic group $\mathbb{D}_{g+1}\times\mathbb{Z}_2\subset O(3)$.\label{2}
\item $\Sigma_g$ has genus $g$.\label{3}
\item $\Sigma_g\to 2H$ in the sense of varifolds as $g\to\infty$.\label{4}
\item For $g$ large,  $\Sigma_g$ has two asymptotically conical ends, $E_1\subset\{z>0\}$ and $E_2\subset\{z<0\}$, which are graphical over $H$ and whose links converge in the $C^0$-topology to the equator $\mathbb{S}^2\cap\{z=0\}$. \label{5}
\item  For any subsequence $g\to\infty$, after taking a further subsequence there exists $r>0$ so that on any compact subset $K$ of $H\setminus C_{r}$, the surfaces $\Sigma_g$ can be expressed as a union of two normal graphs over $K$ that each converges smoothly to $K$ as $g\to\infty$.\label{6}
\end{enumerate}
\end{thm}

\begin{rmk} The numerics of Chopp \cite{C} and Angenent-Chopp-Ilmanen \cite{ACI} (cf. White \cite{W}) suggest already for $g=1$ that $\Sigma_g$ has two ends.  
\end{rmk}

 We conjecture that the surface $\Sigma_g$ minimizes entropy among genus $g$ self-shrinkers.   If true, $\Sigma_g$ could be considered the genus $g$ surface of least geometric complexity in $\mathbb{R}^3$.     In the same way, there is a long-standing conjecture of Kusner \cite{Kusner} that the Lawson surface $\xi_{g,1}$ in $\mathbb{S}^3$ minimizes the Willmore energy among genus $g$ surfaces.   
 
 The existence of an entropy-minimizer among genus $g$ self-shrinkers was obtained by Sun-Wang (Corollary 1.4 in \cite{SW}).  For $g=0$, Bernstein-Wang \cite{BW} (see also \cite{B} and \cite{KZ}) showed that the round sphere minimizes entropy among all embedded two-spheres.  By \cite{ChuSun}, the entropy-minimizer among genus $1$ self-shrinkers is not the Angenent torus\footnote{It is an interesting question whether the genus $1$ self-shrinker obtained in \cite{SW} coincides with $\Sigma_1$.}.


\subsection{Fattening of MCF}
When the flows $\Sigma_g(t) := \sqrt{-t}\Sigma_g$ reach their singularity at $t=0$,  they consist of a double-lobbed cone $\mathcal{C}(\Sigma_g)$.   One may extend the flow for positive times past this time using a weak notion of the mean curvature flow (\cite{CGG}, \cite{ES}).   Such resolutions of cones are often modeled by self-expanders.   A \emph{self-expander} is a hypersurface $\Sigma\subset\mathbb{R}^3$ satisfying
\begin{equation}
H_\Sigma=\frac{\langle x, \mathbf{n}\rangle}{2},
\end{equation}
and give rise to immortal flows by outward homothety $\Sigma_t=\sqrt{t}\Sigma$ for $t\geq 0$.

For $g$ large,  the $t=0$ limit of the MCF beginning at $\Sigma_g$ is a very wide-brimmed double cone.   Angenent-Chopp-Ilmanen \cite{ACI} and Helmensdorfer \cite{H} showed that rotationally symmetric double cones $\{x^2+y^2=\delta z^2\}$ that are sufficiently wide (i.e. $\delta$ is sufficiently small) admit multiple \emph{connected} self-expanding annuli as evolutions\footnote{These are analogous to the stable and unstable catenoids in $\mathbb{R}^3$ bounded between two circles in parallel planes.}.  More generally, Bernstein-Wang (Lemma 8.2 in \cite{BW2}) showed that cones contained in $\{x^2+y^2\leq \delta z^2\}$ for $\delta$ sufficiently small also admit a connected self-expanding evolution.  On the other hand, Ding \cite{D} (based on a sketch of Ilmanen \cite{I}) showed that $\mathcal{C}(\Sigma_g)$ also admits disconnected self-expander evolutions.  

Thus we obtain

\begin{thm}[Fattening]\label{main2}
For $g$ large enough,  the level set flow with (smooth) initial condition given by $\Sigma_g$ fattens.  
\end{thm}

The non-compact examples of Kapouleas-Kleene-M\o ller and Nguyen (\cite{KKM}, \cite{Ng1}, \cite{Ng2}, \cite{Ng3}), on the other hand, blow down to a graphical cone at $t=0$,  and thus have a unique evolution beyond the singular time \cite{EH}.  Some examples in fattening in higher dimensions were considered by Angenent-Ilmanen-Velasquez \cite{AIV}.  



Recently, Lee-Zhao \cite{LZ} showed that for any asymptotically conical self-shrinker, there exists a smooth embedded closed surface that develops that singularity under mean curvature flow. Chodosh-Daniels-Holzgate-Schulze \cite{CDHS} also have shown recently that fattening of a smooth hypersurface with a conical singularity occurs if and only if the cone fattens. Combining these results with our existence result (Theorem \ref{main}) we obtain:

\begin{thm}
There exists a smooth embedded closed surface in $\mathbb{R}^3$ from which the level set flow fattens.  
\end{thm}

While we were completing this article, we learned of work of Ilmanen-White \cite{IW} that constructs (presumably) the same self-shrinkers as those purported in Theorem \ref{main} using mean curvature flow.  In the process Ilmanen-White also obtain examples of mean curvature flows beginning at closed surfaces in $\mathbb{R}^3$ that fatten. Their work uses the recent resolution of the Multiplicity One Conjecture due to Bamler-Kleiner \cite{BK}.  In relation to their work, we get the additional information that the entropy of the self-shrinkers is less than $2$ which may be of independent interest.  On the other hand, Ilmanen-White show their self-shrinkers have two ends for all genera, while we can only show this for large values of the genus.


As a consequence of Bamler-Kleiner's recent resolution of the Multiplicity One Conjecture (together with Brendle's classification of genus zero shrinkers \cite{B} and other recent developments \cite{CHH},\cite{CMC}),  flowing a family of spheres in $\mathbb{R}^3$ is a well-posed problem for which fattening does not occur.  Our work shows that in the higher genus case,  fattening is unavoidable even in the highly restricted class of initial conditions with entropy less than $2$.  

\subsection{Sketch}
Let us sketch the main ideas.  Self-shrinkers are minimal surfaces in the Gaussian metric $(\mathbb{R}^3, e^{-|x|^2/4}\delta_{ij})$ and we will use the corresponding variational construction of such minimal surfaces developed in \cite{KZ}.  Recently with G. Franz and M. Schulz the min-max theory was extended to sweepouts that are equivariant with respect to any finite group $G\subset O(3)$ \cite{FKS}.  

Considering $1$-parameter sweepouts of Gaussian space by spheres invariant under the prismatic group $\mathbb{P}_{g+1} :=\mathbb{D}_{g+1}\times\mathbb{Z}_2$ we obtain from the equivariant min-max theory the self-shrinking sphere $\mathbb{S}^2_*=\mathbb{S}^2(2)$.  After all, the sphere has equivariant index $1$, has lowest entropy shrinker above that of the plane (\cite{BW}, \cite{KZ}) and exists in an optimal $1$-parameter family $\{S_t\}_{t\in [0,1]}$ of concentric spheres centered at the origin.

In \cite{K} the author introduced the notion of ``flipping" optimal foliations of three-manifolds to produce an index $2$ minimal surface (see also \cite{HK} in the free boundary case).  In our setting, we consider the min-max problem associated to all two-parameter families of $\mathbb{P}_{g+1}$-equivariant surfaces of genus $g$ that interpolate between the optimal foliation $\{S_t\}_{t\in [0,1]}$ and the same foliation with the opposite orientation.  The two-parameter family $S_{u,v}$, roughly speaking, consists of two parallel spheres $S_u$ and $S_v$ joined by $g+1$ necks in a $\mathbb{P}_{g+1} :=\mathbb{D}_{g+1}\times\mathbb{Z}_2$ equivariant fashion.   By a Lusternick-Schnirelman argument, a min-max procedure applied to this family does not simply produce $\mathbb{S}^2_*$.   By the catenoid estimate \cite{KMN} the ``width" $\omega_2(\mathbb{P}_{g+1})$ of this family satisfies
\begin{equation}\label{mx}
\omega_2(\mathbb{P}_{g+1}) < 2\lambda(\mathbb{S}_*^2) \approx 2.94.
\end{equation}
 and it would be natural to expect one obtains a doubling of the self-shrinking sphere.  

A Jacobi field argument, however, shows that there can \emph{not} exist doublings of the sphere where necks are congregating only on the equator in the limit that $g\to\infty$.  Indeed,  if $\Sigma_g$ were such a doubling,  by subtracting the distance between the two sheets and rescaling one obtains a rotationally symmetric positive Jacobi field $J=J(\phi)$\footnote{The variable $\phi$ being the azimuthal spherical coordinate.}  on each hemisphere of $\mathbb{S}^2_*$ satisfying \begin{equation}\Delta_{\mathbb{S}^2(1)} J + 4J = 0\end{equation} and $J'(0) = 0$ (to ensure smoothness at the north pole).  On the other hand, $f(\phi) = \cos(\phi)$ (i.e. the $z$-coordinate) satisfies $\Delta_{\mathbb{S}^2(1)} f + 2f = 0$. By Sturm-Liouville comparison, $J$ must have a zero at some $\phi_0< \pi/2$.  This gives a contradiction. Such considerations of ``profile functions" for doublings were introduced by Kapouleas \cite{Kap} and used in his later work \cite{KapMc}.

In fact a large part of this paper is devoted to showing that \eqref{mx} may be improved to \begin{equation}\label{ddd}\omega_2(\mathbb{P}_{g+1})<2.\end{equation} 

The proof \eqref{ddd} relies on a careful study of the metric geometry of Gaussian space.   In particular,   there exists a sweepout with Gaussian areas below $2$ beginning at a given disk $D$ centered about the origin in $H$ and ending at $H\setminus D$,  the exterior infinite annulus, together with a set of arbitrariy small Gaussian area.  The sweepout is obtained by first expanding the disk to a half-cylinder,  and then folding the cylinder through truncated cones until it collapses to $H\setminus D$.  We call such a a sweepout an \emph{inversion} since it reverses the normal vector from upward on $D$ to downward on $H\setminus D$.  The necessary geometric estimates are contained in Propositions \ref{areaellipsoids}, \ref{areacones},  and \ref{gaussianellipsoidsreal}   

Returning to the construction of the desired two-parameter family, we first fold the inner sphere $S_u$ contained in $\{S_{u,v}\}_{(u,v)\in I^2}$ inwards and the outer sphere $S_v$ outwards toward the exterior annulus in $H$ as described above so that the surface resembles parts of two planes parallel to $H$ plus spherical caps of negligible area.   Since $H$ is unstable,  we may then use the catenoid estimate \cite{KMN} to retract the surfaces to the union of two planes (modulo vertical tubes) while maintaining the Gaussian area bound.  Since the two spheres are on separate sides of the plane $H$, we then use Smale's theorem \cite{Smale} to retract them to points in a continuous fashion.  The argument is delicate since the Gaussian areas are forced to approach $2$ in the process of the deformation.  This completes the sketch of the proof of \eqref{ddd}.

We then consider a varifold limit $\Sigma_\infty$ of $\Sigma_g$ for some subsequence of $g$ tending to infinity.  Using work of Kleene-M\o ller \cite{KM} we get that either $\Sigma_{\infty}$ is the self-shrinking cylinder,  the Angenent torus (which has equivariant index $2$ and entropy roughly $1.86$ \cite{Bkog}), an immersed self-shrinker, or $2H$.  We show that the surfaces cannot degenerate to the Angenent torus as the min-max limit is obtained after equivariant neckpinches, all of which reduce the surface to a union of spheres. Immersed surfaces are ruled out by entropy considerations. The self-shrinking cylinder on the other hand has equivariant index $1$ and nullity $0$ and thus cannot arise with multiplicity $1$ from a $2$-parameter min-max procedure.  Indeed, Marques-Neves proved such Morse index bounds in the Almgren-Pitts setting on compact manifolds \cite{MN} and the author together with Liokumovich obtained analogous results in the Simon-Smith setting \cite{KL}.   Having ruled out all other possibilities, we get that $\Sigma_\infty=2H$ and $\Sigma_{g}$ resembles a doubled plane on compact subsets for large $g$.

To show that the genus does not disappear to infinity or vanish into the origin as $g\rightarrow\infty$ we consider the Jacobi equation for the Gaussian area on the plane.   If the topology disappears in the limit, then subtracting the two sheets of the self-shrinker $\Sigma_g$ and doing a blowup argument gives us a ``profile function" satisfying the confluent hypergeometric equation studied by Kummer \cite{Ku} in 1837.   A study of positive radial solutions to this equation prohibits the topology from disappearing in the limit.  

Finally, because the self-shrinkers $\Sigma_g$ consist of components in $\{z>0\}$ and $\{z<0\}$  which are only joined up at a fixed scale where genus is congregrating, we may apply Brakke's regularity theorem to each component separately to deduce that $\Sigma_g$ has only two graphical ends for $g$ large.  In this section we also use ideas of Sun-Wang \cite{SW} on compactness of genus $g$ self-shrinkers.

The organization of this paper is as follows.  In Section 2 we introduce the equivariant min-max setting in which we will work.  In Section \ref{areaestimatessection} we estimate the Gaussian areas of the building blocks of our desired sweepout.  In Section 4 we construct the ``flipping" sweepouts with Gaussian areas less than $2$.  In Section \ref{sectionls} we prove the existence of the purported self-shrinkers $\Sigma_g$.  In Section \ref{genussection} we show that $\Sigma_g$ has genus $g$.  In Section \ref{limitsection} we show that $\Sigma_g$ converges to twice a plane as $g\to\infty$.  In Section \ref{endssection} we show that $\Sigma_g$ has two graphical ends when $g$ is large.   Section \ref{appendix} is an appendix concerned with the self-shrinker Jacobi equation on the plane.
\\
\\
\noindent
\emph{Acknowledgements:}  I am grateful to Prof. Jacob Bernstein for his interest,  encouragement, and many helpful conversations and insights that greatly improved this work.

\section{$\mathbb{P}_{g+1}$-equivariant min-max}
In this section we set up the relevant equivariant min-max framework.
\subsection{Prismatic symmetry} Let us describe the prismatic symmetry group $\mathbb{P}_{g+1}$ in more detail.   If $\mathcal{P}$ is a plane in $\mathbb{R}^3$,  let us denote by $\tau_\mathcal{P}\in O(3)$ the reflection in the plane $\mathcal{P}$.  Recall $H=\{z=0\}$.

Let $\{L_1,..,L_{g+1}\}$ denote $(g+1)$ equally spaced lines contained in $H$ passing through the origin where $L_1$ coincides with the $x$-axis.  In other words for each $k=1,...,g+1$, we denote the rays (in cylindrical coordinates $(r,\theta,z)$ on $\mathbb{R}^3$)
\begin{equation}
R_k =\{(r,\theta,0)\in\mathbb{R}^3 \;|\; \theta = \frac{\pi(k-1)}{g+1} \}
\end{equation}
\noindent
Similarly for each $k=g+2,...,2g+2$ denote the rays
\begin{equation}
R_k =\{(r,\theta,0)\in\mathbb{R}^3 \;|\; \theta = \frac{\pi(k-g-2)}{g+1}+\pi \}
\end{equation}
\noindent
Then for each $k=1,...,g+1$, we have the line $L_k= R_k\cup R_{k+g}$.

The  dihedral group $\mathbb{D}_{g+1}\subset SO(3)$ is generated by a $2\pi/(g+1)$-rotation about the $z$-axis together with $\pi$-rotations about the lines $\{L_i\}_{i=1}^{g+1}$.  Let us denote by $R_i:\mathbb{R}^3\to\mathbb{R}^3$ the rotation about $L_i$.   The group $\mathbb{D}_{g+1}$ has order $2(g+1)$.    

The prismatic group $\mathbb{P}_{g+1}$ is $\mathbb{D}_{g+1}\times\mathbb{Z}_2$, where the $\mathbb{Z}_2$ factor is generated by the reflection \begin{equation}\tau_{H}(x,y,z)= (x,y,-z).\end{equation}

For each $i=1,...,g+1$,  let $P_i$ denote the plane containing the $z$-axis together with the line in $L_i$.  

The group $\mathbb{P}_{g+1}$ contains the $g+1$ vertical reflections $\{\tau_{P_i}\}_{i=1}^{g+1}$ (with $\tau_{P_i}$ obtained as $\tau_H\circ R_i)$ as well as the horizontal reflection $\tau_{H}$. 

\subsection{Equivariant sweepouts}
Let $G\subset O(3)$ be a finite subgroup.   
Set $I^n = [0, 1]^n \in \mathbb{R}^n$.    Let $\{\Sigma_t\}_{t\in I^n}$ be
a family of closed subsets of $\mathbb{R}^3$ and $B\subset\partial I^n$.    We call the family $\{\Sigma_t\}_{t\in I^n}$ a \emph{ genus g $G$-sweepout} if

\begin{enumerate}
\item $g(\Sigma_t) = \Sigma_t$ (setwise) for all $g\in G$ and $t\in I^n$.  
\item  The action of $G$ preserves the orientation of $\Sigma_t$ for $t\in I^n$.
\item $F(\Sigma_t)$ is a continuous function of $t\in I^n$
\item  $\Sigma_t$ converges to $\Sigma_{t_0}$ in the Hausdorff topology as $t\rightarrow t_0$.  
\item For $t_0 \in I^n\setminus B$, $\Sigma_{t_0}$ is a smooth closed surface of genus g and $\Sigma_t$ varies smoothly for $t$ near $t_0$.
\item For $t\in B$,  the set $\Sigma_t$ consists of the union of a 1-complex (possibly empty) together with a smooth surface (possibly empty).
\end{enumerate}
\begin{rmk}
By item (2),  a $G$-invariant sweepout is orthogonal to (or disjoint from) any plane of reflective symmetry.
\end{rmk}

Beginning with a genus g sweepout $\{\Sigma_t\}_{t\in I^n}$ we need to construct comparison sweepouts which agree with $\{\Sigma_t\}_{t\in I^n}$ on $\partial I^n$.   For any sweepout, we may generate new swepouts as follows.  For any map  $\Psi\in C^\infty (I^n\times \mathbb{R}^3, \mathbb{R}^3)$ such that for all $t\in I^n$ we have $\Psi(t,.)\in \mbox{Diff}_0(\mathbb{R}^3)$ and $\Psi(t,.) = id$ if $t\in \partial I^n$.  If $G\neq \{e\}$ we further demand the family of diffeomorphisms be $G$-equivariant. If $\Pi$ is a collection of sweepouts, we say it is \emph{saturated} if for any any sweepout $\{\Lambda_t\}_{t\in I^n}\in\Pi$ we have also $\Psi(t,\Lambda_t)\in \Pi$ as long as $\Psi$ has one of the two additional properties:
\begin{enumerate}
\item  $\Psi(t,.):\mathbb{R}^3\to\mathbb{R}^3$ has compact support for all $t\in I^n$.
\item $\Psi(t,.)$ is the time $1$ flow generated by smooth $n$-parameter vector fields $X_t:\mathbb{R}^3\to\mathbb{R}^3$ with $\sup_{t\in I^n}||X_t|_{C^1}\leq C$.
\end{enumerate}

 Given a sweepout $\{\Sigma_t\}_{t\in I^n}$, denote by $\Pi := \Pi_{\Sigma_t}$ the smallest saturated collection of sweepouts containing $\{\Sigma_t\}_{t\in I^n}$  We define the \emph{width} of $\Pi$ to be

\begin{equation}
W(\Pi,G) = \inf_{\Lambda_t\in \Pi} \sup_{t\in I^n} F(\Lambda_t).
\end{equation}

A \emph{minimizing sequence} is a sequence of sweepouts $\{\Sigma^i_t\}\in\Pi$ such that
\begin{equation}
\lim_{i\rightarrow\infty} \sup_{t\in I^n} F(\Sigma_t^i) = W(\Pi,G).
\end{equation}

Finally, a \emph{min-max sequence} is a sequence of slices $\Sigma^i_{t_i}$,  $t_i\in I^n$ taken from a minimizing
sequence so that \begin{equation}F(\Sigma_{t_i}^i)\rightarrow W(\Pi, G). \end{equation}  

The main point of the Min-Max Theory of Almgren-Pitts (\cite{Al}, \cite{Pitts}) as refined by Simon-Smith (\cite{SS} \cite{DP}) is that if the width is greater than the maximum of the areas of the boundary surfaces, then some min-max sequence converges to a minimal surface in $M$.  Crucially, the topology of the limit is controlled by $g$.  

The equivariant version was obtained in \cite{K} (cf. \cite{PR}) assuming the elements of $G$ are orientation-preserving, and recently with G. Franz and M. Schulz \cite{FKS} this assumption was removed.  The extension to the Gaussian metric $M=(\mathbb{R}^3,e^{-|x|^2/4}\delta_{ij})$ was obtained together with X. Zhou \cite{KZ} (see also \cite{SWZ}) using a special flow to prevent min-max limits from escaping to infinity.   

\begin{thm} [Multi-parameter Min-max Theorem]\label{highparamminmax}
 Given a (genus g) $G$-sweepout $\{\Sigma_t\}_{t\in I^n}$, if $\Pi$ denotes its saturation then the following holds.  If
\begin{equation}\label{isbigger}
2>W(\Pi,G)> \sup_{t\in \partial I^n} F(\Sigma_t),
\end{equation}
then there exists a min-max sequence  $\Sigma_i := \Sigma^i_{t_i}$ such that 
\begin{equation}
\Sigma_i \rightarrow \Gamma\mbox{ as varifolds, }
\end{equation} 
where $\Gamma$ is a smooth embedded $G$-equivariant self-shrinker.  Moreover,
\begin{equation}
\lambda(\Sigma) = W(\Pi, G) = F(\Sigma). \end{equation}
\noindent
We also have
\begin{enumerate} 
\item 
$genus(\Gamma)\leq g$.
\item The self-shrinker $\Gamma$ is orthogonal to any plane and axis of symmetry of $G$ that it intersects.
\end{enumerate}
\end{thm}

We restrict to sweepouts with entropy less than $2$ because it rules out higher multiplicity and simplifies statements about the genus control.  Given the recent work of Wang-Zhou \cite{WZ} one expects multiplicity does not occur in this setting.

As an important example of Theorem \ref{highparamminmax},  for each $g\geq 1$, we have the following (genus $0$) $\mathbb{P}_{g+1}$-sweepout $\{S_t\}_{t\in [0,1]}$ of Gaussian space:
\begin{equation}\label{spheres}
S_t  = \{(x,y,z)\;|\; x^2+y^2+z^2 = \tan(\frac{t\pi}{2})\}.
\end{equation}

If we apply Theorem \ref{highparamminmax} to the saturation of the family \eqref{spheres} the self-shrinker we obtain is $\mathbb{S}^2_*$.  Indeed, by item (1) above the genus must be zero, and by Brendle's classification \cite{B} we obtain only the sphere, cylinder or plane as possibilities.  The latter is excluded by item (2) and the sphere has smaller Gaussian area.  

\subsection{$\mathbb{P}_\infty$-invariant self-shrinkers}

We will need the following classification, due to Kleene-M\o ller (Theorem 2 in \cite{KM}):

\begin{prop}[Classification of $\mathbb{P}_\infty$-invariant shrinkers]\label{classification}
A smooth $\mathbb{P}_\infty$-invariant embedded self-shrinker is one of the following:
\begin{enumerate}[(i)]
\item $\mathbb{S}_*^2$
\item $\mathbb{S}_*^1\times\mathbb{R}$ (about the $z$-axis)
\item  Angenent torus 
\item $H$ (the $xy$-plane).
\end{enumerate}
\end{prop}

\subsection{Equivariant index of self-shrinkers}
Let $L_{0,1}$ denote the stability operator for the $F_{0,1}$ functional defined on any self-shrinker $\Sigma$.   Let us call $\phi\in C^\infty(\Sigma)$,  a \emph{$\mathbb{P}_{g+1}$-invariant eigenfunction} if
\begin{enumerate}
\item $L_{0,1}\phi = \lambda\phi \mbox{ for some } \lambda\in\mathbb{R}$
\item $\phi(hx) = \phi(x)\mbox{ for } h\in\mathbb{P}_{g+1}\mbox{ and }x\in\Sigma$.
\end{enumerate}

Let $\mbox{ind}_{\mathbb{P}_{g+1}}(\Sigma)$ denote the maximal dimension of a space of eigenfunctions of $L_{0,1}$ with negative eigenvalue (counted with multiplicities),  and let $\mbox{null}_{\mathbb{P}_{g+1}}(\Sigma)$ denote the dimension of the $\mathbb{P}_{g+1}$-invariant kernel of $L_{0,1}$.  



The positive and negative eigenspaces of $L_{0,1}$ on spheres and cylinders are easy to calculate.  We obtain
\begin{prop}[Equivariant index]\label{indexbounds}
We have the following
\begin{enumerate}
\item $\mbox{ind}_{\;\mathbb{P}_{g+1}}(\mathbb{S}^2_*) = 1$ and $\mbox{null}_{\;\mathbb{P}_{g+1}}(\mathbb{S}^2_*) = 0$,
\item $\mbox{ind}_{\;\mathbb{P}_{g+1}}(\mathbb{S}^1_*\times\mathbb{R}) = 1$ and $\mbox{null}_{\;\mathbb{P}_{g+1}}(\mathbb{S}^1_*\times\mathbb{R}) = 0$,
\end{enumerate}
\end{prop}

\begin{proof}
The negative eigenfunctions for the stability operator for the $F_{0,1}$ functional on $\mathbb{S}^2_*$,  $\mathbb{S}^1_*\times\mathbb{R}$ correspond to translations and homothety (Lemma 5.5 in \cite{CM2}).  The eigenfunctions with zero eigenvalue (nullity) correspond to rotations.  The sphere is preserved by all rotations and no translation is $\mathbb{P}_{g+1}$ invariant, leaving only the homothety as negative eigenfunction and no nullity.  Similarly no translation or rotation of the cylinder is $\mathbb{P}_{g+1}$ invariant, leaving only the homothety as $\mathbb{P}_{g+1}$-invariant negative eigenfunction.\end{proof}



\section{Gaussian area estimates}\label{areaestimatessection}
In this section we derive a number of estimates for the Gaussian areas of standard surfaces that will comprise the building blocks of our family.   All surfaces we consider are $\mathbb{P}_\infty$-invariant. 

For $0\leq R_1\leq R_2$ let us denote the doubled annulus:
\begin{equation}
D(R_1,R_2, h):= \{(x,y,z)\in\mathbb{R}^3\;|\; R^2_1\leq x^2+y^2\leq R_2^2 \mbox{ and } z=\pm h\}.
\end{equation}
\noindent
The Gaussian area can be computed\footnote{With the convention that $|D(R_1,R_2,0)|$ counts the area twice.} as
\begin{equation}\label{gaussianareadisk}
|D(R_1,R_2, h)| = 2e^{-h^2/4}(e^{-R_1^2/4}-e^{-R_2^2/4}). 
\end{equation}
\noindent
Let us denote the cylinder:
\begin{equation}
Cyl(r,h) = \{(x,y,z)\in\mathbb{R}^3\;|\; x^2+y^2= r^2 \mbox{ and } |z|\leq h\}, 
\end{equation}
\noindent
with Gaussian area
\begin{equation}
Cyl(R,H) = Re^{-\frac{1}{4}R^2}\int_0^H e^{-\frac{1}{4}z^2}dz, 
\end{equation}
so that 
\begin{equation}\label{cyl}
Cyl(R,\infty) =\sqrt{\pi}Re^{-\frac{1}{4}R^2}, 
\end{equation}
and
\begin{equation}\label{allcyl}
Cyl(R,h)\leq hRe^{-R^2/4} \mbox{ for all } h \mbox{ and } R.
\end{equation}
The sphere
\begin{equation}
S(r) = \{(x,y,z)\in\mathbb{R}^3\;|\; x^2+y^2+z^2= r^2\}, 
\end{equation}
satisfies

\begin{equation}\label{sphere}
|S(R)| = R^2 e^{-\frac{1}{4}R^2}.
\end{equation}

The following lemma is useful. 
\begin{lemma}\label{easy}
Fix $h>0$ and $\Sigma$ a submanifold (possibly with boundary) contained in $\{z\geq 0\}$.   Consider the translated surface $\Sigma+h\vec{k}$, where $\vec{k}= (0,0,1)$.  Then
\begin{equation}
F(\Sigma+h\vec{k}) \leq e^{-h^2/4}F(\Sigma).  
\end{equation}
\end{lemma}

\begin{proof}
Indeed, (writing $|\rho|^2=x^2+y^2+z^2$) we get
\begin{equation}
F(\Sigma+h\vec{k}) = \int_{\Sigma+h} e^{-|\rho|^2/4}d\mu = \int_\Sigma e^{\frac{=|\rho|^2-2hz- h^2}{4}}d\mu \leq e^{-h^2/4}\int_\Sigma e^{-|\rho|^2/4}d\mu,
\end{equation}
where we use that $z\geq 0$ on the support of $\Sigma$ and $h$ is positive in the inequality. 
\end{proof} 

\subsubsection{Gaussian area of spherically-capped cylinders}
Let $S(R,h)$ denote the translated spherical caps:
\begin{equation} 
S(R,h) = \{(r,\theta,z)\in\mathbb{R}^3\;|\; z=\pm\sqrt{R^2-r^2}\pm h\}.
\end{equation}
In the following,  we estimate the Gaussian surface area of ellipsoid-like surfaces obtained from adding in a cylinder to $S(R,h)$ to obtain a closed surface.

\begin{prop}[Gaussian area of spherically-capped cylinders]\label{areaellipsoids}
For $\delta_1=.133$ there holds for all $R>0$ and $h>0$ 
\begin{equation}\label{pointy}
|S(R,h)|+|Cyl(R,h)|\leq 1.867 = 2 - \delta_1.
\end{equation}
\end{prop}

\begin{proof}
Using \eqref{cyl} and \eqref{sphere} together with Lemma \ref{easy} we obtain
 \begin{equation}
|S(R,h)|+|Cyl(R,h)|\leq e^{-h^2/4}R^2e^{-R^2/4} + Re^{-R^2/4}\int_0^h e^{-1/4z^2}dz = G(R,h).
\end{equation}
For fixed $R$, the derivative of $G(R,h)$ is zero only at $h = 2/R$.  Thus for each $R>0$ we obtain
\begin{equation}\label{fixedR}
\sup_{h\geq 0} |S(R,h)|+|Cyl(R,h)|\leq \max{(G(R,2/R),  |S(R,0)|,  |Cyl(R,\infty)|)}
\end{equation}
where
\begin{equation}\label{form}
G(R,2/R) = R^2e^{-R^2/4-1/R^2}+Re^{-R^2/4}\int_0^{2/R} e^{-z^2/4}dz.
\end{equation}
We find numerically the maximum of the function $F(R) = G(R,2/R)$ is at most $1.867$ and achieved at $R_0\approx 1.764$.    Thus we obtain from \eqref{fixedR}:
\begin{equation}
\sup_{R,h} |S(R,h)|+|Cyl(R,h)|\leq \max{(1.867,  \lambda(\mathbb{S}_*^2), \lambda(\mathbb{S}_*^1\times\mathbb{R}))}=1.867
\end{equation}
Thus \eqref{pointy} follows with $\delta_1=2-G(2/R_0,R_0)\leq .133$.  

\end{proof}

\subsubsection{Gaussian area of cones}

For $\phi\in [0,\pi/2]$ and $0\leq R_1<R_2$,  consider the graph of a (doubled) cone in cylindrical coordinates:
\begin{equation}
C(R_1,R_2,\phi) = \{(r,\theta,z)\in\mathbb{R}^3\; |\; z=\pm \tan(\phi)(r-R_1)\mbox{ and } R_1<r<R_2\}
\end{equation}
We compute the Gaussian area by
\begin{equation}\label{cones}
|C(R_1,R_2,\phi)| = \sec(\phi)\int_{R_1}^{R_2} r e^{-\frac{1}{4}(r^2+\tan^2(\phi)(r-R_1)^2} dr.
\end{equation}
Note that setting $\phi=0$,  we recover the formula for Gaussian area of the doubled annuli:
\begin{equation}
|C(R_1,R_2,0)|=|D(R_1,R_2,0)| = 2(e^{-\frac{1}{4}R_1^2} -e^{-\frac{1}{4}R_2^2} ).
\end{equation}


Carrying out the integral in \eqref{cones} we obtain
\begin{equation}
|C(R,\infty,\phi)|= 2\cos(\phi)e^{-\frac{1}{4}R^2}+I(R,\phi), 
\end{equation}
where $I(r,\phi)\geq 0$ is given by\footnote{Recall $\mathrm{erf}(x):=\frac{2}{\sqrt{\pi}}\int_0^xe^{-y^2}dy$, normalized so that $\mathrm{erf}(\infty)=1.$}
\begin{equation}
I(R,\phi)=\sqrt{\pi}R\sin^2(\phi)(1-\mathrm{erf}(\frac{R}{2}\cos(\phi)))e^{-\frac{R^2\sin^2(\phi)}{4}}.
\end{equation}

Differentiating $I(R,\phi)$ in $\phi$ (and discarding the negative terms), we obtain
\begin{equation}
I_\phi(R,\phi)\leq\sqrt{\pi} R\sin(\phi)
\end{equation}
so that
\begin{equation}\label{derivee}
|C_\phi(R,\infty,\phi)|\leq 2e^{-\frac{R^2}{4}}\sin(\phi)(-1+\sqrt{\pi} Re^{\frac{R^2}{4}}).
\end{equation}
Taking $R\leq 0.2$ we obtain \begin{equation}\sqrt{\pi} Re^{\frac{R^2}{4}}\leq 1/2.\end{equation} Thus by \eqref{derivee} we guarantee
\begin{equation}
|C_\phi(R,\infty,\phi)|<0 \mbox { for all } \phi\in(0,\pi/2].
\end{equation}
Thus we have $|C(R,\infty,\phi)|$ is decreasing in $\phi$ for $R$ small enough.   In fact, we have the following:

\begin{prop}[Gaussian area of cones]\label{conesbound}
For $\delta_2=.02>0$ the following is true:
\begin{enumerate}
\item For all $R_1\leq R_2$ and $\phi\in [0,\pi/2]$ there holds
\begin{equation}
|C(R_1,R_2,\phi)|\leq 2. 
\end{equation}
\item For $R_1>.2$ there holds
\begin{equation}\label{delta2}
|C(R_1,\infty,\phi)|\leq 2 - \delta_2.
\end{equation}
\item
For $R_1\leq .2$,  $|C(R_1,\infty, \phi)|$ is strictly decreasing in $ \phi\in[0,\pi/2]$, whence
\begin{equation}
 |C(R_1,\infty, \phi)|\leq |D(R_1,\infty)|\mbox{ for all }  \phi\in[0,\pi/2].
\end{equation}
\end{enumerate}
\end{prop}
\begin{proof}
We have already shown (3).  To show (2), it suffices to find the maximum of $|C(R,\infty,\phi)|$ for $(R,\phi)$ in the domain $[.2,\infty]\times[0,\pi/2]$.    One finds numerically that the maximum is achieved at approximately $1.98$ at the boundary point $(R,\phi) = (.2, 0)$.   Thus \eqref{delta2} holds with $\delta_2=.02$.\footnote{This is far from sharp, but for our purposes the existence of $\delta_2>0$ is sufficient.}
\end{proof}
\begin{rmk}
For $R$ small,  $|C(R,\infty,\phi)|$ is decreasing in $\phi$ and for $R$ large $|C(R,\infty,\phi)|$ is increasing in $\phi$.   In the intermediate region $|C(R,\infty,\phi)|$ has one critical point in $\phi$ for each $R$.  This behavior is consistent with the existence of self-shrinking conical ``trumpet ends" discovered by Kleene-M\o ller \cite{KM} that interpolate between the cylinder and (doubled) plane. 
  
\end{rmk}

In fact we will need to consider the translations of the cones.  Considering the translated cone $C(R_1,\infty,h, \phi)$ defined over $R_1<r<R_2$ by
\begin{equation}\label{liftedcones}
C(R_1,R_2, h,\phi) = \{(r,\theta,z)\; |\; z=\pm \tan(\phi)(r-R_1)\pm h\mbox{ and } R_1<r<R_2\}
\end{equation}

Using Lemma \ref{easy} we can estimate the Gaussian area
\begin{equation}\label{simple}
|C(R_1,R_2,h, \phi)|\leq e^{-\frac{h^2}{4}} |C(R_1,R_2,\phi)|.
\end{equation}

Thus from Proposition \ref{conesbound} we obtain:

\begin{prop}[Gaussian area of translated cones]\label{areacones}
For $\delta_2$ defined in Proposition \ref{conesbound},  the following are true:
\begin{enumerate}
\item For $R>.2$ and all $h$ there holds
\begin{equation}
|C(R,\infty,h,\phi)|\leq 2 - \delta_2.
\end{equation}
\item
For $R\leq .2$ and all $h$ and $\phi\in [0,\pi/2]$ there holds
\begin{equation}
|C(R,\infty,h,\phi)|\leq |D(R,\infty,h)|.
\end{equation}
\end{enumerate}
\end{prop}

\begin{proof}
For $R\geq .2$,  we may bound (for all $h$) using Lemma \ref{easy}
\begin{equation}
|C(R,\infty,h,\phi)|\leq e^{-h^2/4}(2-\delta_2) \leq   2-\delta_2.
\end{equation}

To show (2), observe that for $R\leq .2$ by item (3) in Proposition \ref{conesbound} together with Lemma \ref{easy} we obtain
\begin{equation}
|C(R,\infty,h,\phi)|\leq e^{-h^2/4} |C(R,\infty, \phi)| \leq e^{-h^2/4}|C(R,\infty, 0)| = |D(R,\infty,h)|.
\end{equation}

\end{proof}





\subsubsection{Gaussian areas of ellipsoids}
Consider the ellipsoid with $a>b$ given by 
\begin{equation}
E(a,b)= \{(x,y,z)\in\mathbb{R}^3\; |\; (x^2+y^2)a^{-2} + z^2b^{-2} = 1\}
\end{equation}
The formula for the Gaussian area of $E(a,b)$ is 
\begin{equation}\label{ellipsoidformula}
|E(a, b)| = e^{-\frac{a^2}{4}}\int_0^1 e^{\frac{1}{4}(a^2-b^2)\tau^2}a^2\tau \sqrt{1+b^2a^{-2}\tau^{-2}(1-\tau^2)}d\tau. 
\end{equation}


Note that $E(a,0)$ is (in the sense of varifolds) the disk $D(0,a)$ with multiplicity $2$ and $E(a,a)$ is the sphere $S(a)$.  
We have the following confirmed numerically:

\begin{prop}[Gaussian areas of ellipsoids]\label{gaussianellipsoidsreal}
Setting $\delta_3=.0365>0$, for all $b\leq a\leq 4$ there holds
\begin{equation}\label{ellipsoidaway}
|E(a,b)|\leq 2-\delta_3= 1.9365.
\end{equation}
\end{prop}


We will also need to consider the following radial graphs \begin{equation}z^+_{h,a,b}: H\to \mathbb{R}^+\end{equation} for each $h\leq b\leq a$:
\begin{equation}
z^{+}_{h,a,b}(r,\theta)  = 
\begin{cases}\label{combinedgraphs}
 b\sqrt{1-a^{-2}r^2}  & \mbox{ for            } r\leq a\sqrt{1-h^2b^{-2}} ,\\
h &\mbox{          for        } r\geq a\sqrt{1-h^2b^{-2}} ,\\
\end{cases}
\end{equation}

Define $z^-_{h,a,b}(r,\theta)= -z^+_{h,a,b}(r,\theta)$.  The family is obtained by taking the maximum of the graphs determined by $z=h$ and $E(a,b)$.  Fixing $h$ and $a$, we obtain a $1$-parameter family of graphs $\{z^+_{h,a,b}\}_{b=h}^a$ interpolating from the constant function $z^+_{h,a,h}=h$ at $b=h$ to the union of a piece of sphere glued together with constant function at $b=a$.  This family effectively removes the parts of the ellipsoid $E(a,b)$ with vertical tangent line.  

Let $z_{h,a,b} =  z^+_{h,a,b}\cup z^-_{h,a,b}$.  We need the following estimate:

\begin{lemma}\label{graphs}

Suppose $a\geq 3$.  Then for all $0\leq h\leq b\leq a$ there holds
\begin{equation}\label{bestbound}
|z^+_{h,a,b}|\leq 1-\frac{h^2}{4}.
\end{equation}

\end{lemma}

\begin{proof}
For the elliptical part of $z^+_{h,a,b}$, the area may be computed by restricting the integral in \eqref{ellipsoidformula} to $\tau\in [h/b,1]$.  The formula for the part of $z^+_{h,a,b}$ which is constant is given in \eqref{gaussianareadisk}.  Numerically we find that $|z^+_{h,a,b}|$ is decreasing in $b$ as long as $a\geq 3$.  When $b=h$, we obtain equality in \eqref{bestbound}.
\end{proof}


\section{Existence $\mathbb{P}_{g+1}$-equivariant flipping sweepouts}
In this section,  we construct the non-trivial two-parameter ``flipping" sweepout that will be used to produce the self-shrinker $\Sigma_g$.  In the following, let $V$ denote the volume of Gaussian space\footnote{$V=\frac{1}{(4\pi)^{3/2}}\int_{\mathbb{R}^3} e^{-3|\rho|^2/8}dx = \frac{1}{(4\pi)^{1/2}}\int_0^\infty r^2e^{-3r^2/8}dr\approx .5445$.}.  We prove the following (recall $I=[0,1]$ and the optimal family of spheres $\{S_t\}_{t\in [0,1]}$ was defined in \eqref{spheres}):
\begin{thm}[Flipping sweepout]\label{optimalsweepout}
Fix $g\geq 1$.  There exists a (genus $g$) $\mathbb{P}_{g+1}$-sweepout $\{\Sigma_{s,t}\}_{(s,t)\in I^2}$ so that the following hold for suitably small $\eta_1\ll 1$ and $\eta_2<\frac{V}{2}$.
\begin{enumerate}
\item $\Sigma_{0,t}=S_{1-t}$ together with a set of arcs for $t\in [0,1-\eta_1]$ \label{a} 
\item $|\Sigma_{0,t}|\leq \eta_1$ for $t\in [1-\eta_1,1]$ \label{b}
\item $\Sigma_{1,t} = S_{t}$ together with a set of arcs for $t\in [\eta_1,1-\eta_1]$ \label{c}
\item $|\Sigma_{1,t}|\leq \eta_1$ for $t\in I\setminus [\eta_1,1-\eta_1]$ \label{d}
\item $|\Sigma_{s,0}|+|\Sigma_{s,1}|\leq \eta_1$ for all $s\in I$ \label{ee}
\item $\Sigma_{0,t}$ has the opposite orientation to $\Sigma_{1,1-t}$ for $t\in [\eta_1,1-\eta_1]$. \label{opposite}  
\item $\Sigma_{s,t}$ has genus $g$ for $(t,s)\notin \partial I^2$. \label{f}
\item For each $s\in[0,1]$, the $1$-sweepout $\{\Sigma_{s,t}\}_{t\in [0,1]}$ is a non-trivial sweepout in the following sense.  For some continuously varying choice of region $R_{s,t}$ bounded by $\Sigma_{s,t}$, there holds: $\mbox{vol} (R_{s,0})\leq\eta_2$ and $\mbox{vol}(R_{s,1})>V-\eta_2$ for each $s\in [0,1]$. \label{great}
\item The Gaussian area bound holds: \begin{equation} \sup_{(s,t)\in\partial I^2} F(\Sigma_{t,s})= \lambda(\mathbb{S}^2_*).\end{equation}
where the supremum is attained at a unique point on both the left $\{(0,t)\}_{t\in [0,1]}$ and right $\{{(1,t)}\}_{t\in [0,1]}$ sides of the parameter space $I^2$.
\item The Gaussian area bound holds: \begin{equation} \sup_{(s,t)\in I^2} F(\Sigma_{s,t})<2.\end{equation}
\end{enumerate}
\end{thm}

\subsection{Inversions at varying radii}
In this subsection,  we begin to construct the family posited in Theorem \ref{optimalsweepout}.   For $R$ in a suitable range  we will construct a family of surfaces  $\{\Sigma^R_t\}_{t\in [0,5/6]}$.   We call $\{\Sigma^R_t\}_{t\in [0,5/6]}$ an `inversion at radius $R$' because it first consists of a spherical foliation (modulo thin tubes connecting to a negligible interior sphere) beginning at radius infinity and stopping at radius $R$ after which the sphere is ``inverted" toward the part of a plane parallel to the $xy$-plane with radial coordinate bigger than $R$ (together with large spherical caps).  This process involves interpolation with cylinders and cones as described in the introduction.  The inversion process brings the Gaussian areas extremely close to, but below, $2$.

The family of surfaces $\{\Sigma^R_t\}_{t\in [0,5/6]}$ depends on several auxiliary variables: \begin{equation}\Sigma^R_t=\Sigma^R_t(h(R),\delta(R), \Omega(R,h))\end{equation} where $h, \delta, \Omega$ will be chosen in the course of the construction.  Roughly speaking, the parameter $\delta$ will be very small and represents the thickness of a certain collection of tubes.  The parameter $h$ will represent the radius and height of the surface $F_h(h)$ (i.e.  a cylinder capped with disks).  The parameter $\Omega(R, h)$ is large and measures to what radius we fling out a sphere toward infinity in the deformation.  The sweepout is obtained by concatenating the following five isotopies.  

For ease of notation, in this section for $\Sigma\subset\mathbb{R}^3$ a piecewise smooth surface we denote $F(\Sigma)$ by $|\Sigma|$.
\\
\\
Let us set
\begin{equation}\label{ranger}
.2 \leq R \leq 5.
\end{equation}
{\it Step 0: Set-up and definitions.}
For any $h,r>0$ let 
\begin{equation}\label{defofsmall}
F_h(r) = Cyl(r,h)\cup D(0,r,h), 
\end{equation}
denote the sphere obtained by capping off a cylinder with disks.  
Assume $\epsilon\ll\min({h,r})$. Let $L(\epsilon)$ denote the boundary of the $\epsilon$-tubular neighborhood about the union of the rays $\cup_{i=1,odd}^{2g+1}R_i$ \footnote{When $g$ is odd,  the union includes half-lines in opposite directions, giving a union of lines.}.  Let $D^1(h, r, \epsilon)$ denote the union of the $g+1$ disks in $Cyl(r,h)$ bounded by the $g+1$ circles \begin{equation}L(\epsilon)\cap Cyl(r,h)\subset F_h(r).\end{equation}  Similarly let $D^2(R, \epsilon)$ denote the union of disks in $S(R)$ bounded by the $g+1$ small circles $L(\epsilon)\cap S(R)$.  Note that for $\epsilon\ll\min({h,r})$ there holds
\begin{equation}
|D^1(h, r, \epsilon)|\leq 2\pi\epsilon^2 \mbox{ and } |D^2(R, \epsilon)|\leq 2\pi\epsilon^2
\end{equation}

Let $L(R,h,r, \epsilon)$ denote the union of $g+1$ annuli comprising the connected components of $L(\epsilon)\setminus (Cyl(r,h)\cup S(R))$ with boundary $\partial D^1(h,r, \epsilon)\cup \partial D^2(R,\epsilon)$.  Since straight lines through the origin in Gaussian space to infinity have bounded length, there exists $A>0$ so that
\begin{equation}\label{length}
|L(R,h,r,\epsilon)|\leq |L(\epsilon)|\leq A\epsilon.
\end{equation}

\noindent 
{\it Step 1: Initial surfaces.}
\noindent 
For $h<.1$ let us denote
\begin{equation}\Sigma_0(R,h)=S(R)\cup F_h(h). 
\end{equation}

We define a smoothly varying family of surfaces for $t\in (0,1]$: 

\begin{equation}
\Sigma_1^R(t) = \Sigma_0(\frac{R}{t},ht)\cup(L(\frac{R}{t}, th,th,th^3))\setminus (D^1(th,th,th^3)\cup D^2(\frac{R}{t},  th^3)).
\end{equation}
\noindent
At $t=1$ we get
\begin{equation}
\Sigma_1^R(1)= \Sigma_0(R,h)\cup L(R, h,h,h^3)\setminus (D^1(h,h,\epsilon)\cup D^2(R,  h^3)), 
\end{equation}
\noindent
and as $t\to 0$, the family $\Sigma_1(t)$ converges in the Hausdorff topology to the union $\cup_{i=1,odd}^{2g+1}R_i$.  

Roughly speaking, for $t\in (0,1)$, the surface $\Sigma_1(t)$ is a larger sphere with a tiny cylinder inside capped off with disks, together with $g+1$ thin necks joining the inner component to the outer sphere.   At $t=0$ the set $\Sigma^R_1(0)$ is a union of $g+1$ equally spaced rays through the origin.

By Taylor expanding the Gaussian area of $F_h(h)$ we see that there exists $h_0>0$ and $B>0$ so that if $h\leq h_0$, there holds 
\begin{equation}\label{taylorlittle}
|F_h(h)|\leq Bh^2.
\end{equation}
By plugging in $\epsilon=th^3$ into \eqref{length} we obtain
\begin{equation}
|\Sigma^R_1(t)|\leq |S(R)| + |F_h(h)| + L(\epsilon) \leq \lambda(S^2_*)+ Bh^2 + Ah^3.
\end{equation}
Choose $h_1$ so that \begin{equation}Bh_1^2 + Ah_1^3<2-\lambda(S^2_*).\end{equation}  Then as long as $h\leq\min(h_1, h_0,.1)$ we obtain:
\begin{equation}
\sup_{t\in [0,1]} |\Sigma_1^R(t)| < 2.
\end{equation}
\noindent
{\it Step 2: Outer sphere to cylinders with caps.}
Fix $\Omega=\Omega(R)>0$ (it will be a large number to be determined later).  For $t\in [0,1]$ let us define
\begin{equation}
G(R,t)  = Cyl(R,\Omega t)\cup S(R,\Omega t).
\end{equation}
Note that $G(R,0)$ coincides with $S(R)$ and $G(R,1)$ is nearly a cylinder a long as $\Omega$ is sufficiently large. 

By Proposition \ref{areaellipsoids} there exists $\delta_1>0$ so that
\begin{equation}\label{good}
|G(R,t)| < 2-\delta_1 \mbox{ for all } t\in [0,1].
\end{equation}
Assume $\epsilon\ll \min(h,r)$.  Let $\tilde{D}^2(R, \epsilon, t)$ denote the small disks in  $G(R,t)$  bounded by the curves of intersection of $L(\epsilon)$ with $G(R,t)$ and let $\tilde{L}(R,h,r,\epsilon)$ denote the component of $L(\epsilon)\setminus (Cyl(r,h)\cup G(R,t))$ with boundary given by $\partial D^1(h,r,\epsilon)\cup \partial \tilde{D}^2(R,\Omega, \epsilon, t)$.    

We define a continuously varying one-parameter family $\{\Sigma_2^R(t)\}_{t\in [0,1]}$ with $\Sigma_2^R(0) = \Sigma_1^R(1)$ as follows:
\begin{equation}
\Sigma_2^R(t) = G(R,t)\cup F_h(h) \cup \tilde{L}(R,h,h,h^3))\setminus (D^1(h,h,h^3)\cup \tilde{D}^2(R,\Omega,h^3,t)). 
\end{equation}
Recall that for $h$ sufficiently small, we have
\begin{equation}\label{de}
|\tilde{L}(R,h,h,h^3)|\leq |L(\epsilon)|\leq Ah^3.
\end{equation}
In light of \eqref{good} and \eqref{de} and \eqref{taylorlittle} as long as $h\leq\min(h_0,h_1,.1)$ we obtain (for all $R$ in the desired range \eqref{ranger} and all $\Omega>0$):  
\begin{equation}
\sup_{t\in [0,1]} |\Sigma_2^R(t)| <2 - \delta_1+Ah^3+Bh^2.
\end{equation}
Choose $h_2$ to satisfy
\begin{equation}
Ah_2^3+Bh_2^2\leq \delta_1/2.
\end{equation}
Then as long as $h\leq\min(h_0,h_1,h_2,.1)$ we obtain
\begin{equation}
\sup_{t\in [0,1]} |\Sigma_2^R(t)| <2 - \delta_1/2.
\end{equation}

The variable $\Omega=\Omega(R)$ (i.e. the height of the cylinder) will be specified in the next two steps.  
\\
\\
{\it Step 3: Opening via truncated cones.}

For $t\in [0,1]$ let us set (recalling definition \ref{liftedcones})
\begin{equation}
G(R,\Omega, h, t) = C(R, R+\sin(\frac{t\pi}{2})(\Omega-h), h, \frac{t\pi}{2})\cup Ends(R,\Omega,h,t)\cup Cyl(R,h)
\end{equation} 
where $Ends(R,\Omega,h,t)$ is given as 
\begin{equation}
Ends(R,\Omega,h,t) = S(R, \Omega)\cup B(R,\Omega,h,t)
\end{equation}
with $B(R,\Omega,h,t)$ being the set (far from the plane $H$) traced out as the cone opens:
\begin{equation}
B(R,\Omega,h,t) = \{|z|>h\}\cap (\bigcup_{s\in [0,t]} \partial C(R, R+\sin(\frac{s\pi}{2})(\Omega-h), h, \frac{s\pi}{2}).
\end{equation}
\noindent
Assume $\Omega\geq R$.  Then all points in the support of $S(R,\Omega)$ are a distance at least $\Omega$ from the origin.  Thus we may bound
\begin{equation}\label{outthere1}
|S(R,\Omega)|\leq \frac{e^{-\Omega^2/4}area_{\mathbb{R}^3}(S(R))}{4\pi} \leq R^2e^{-\Omega^2/4} \leq\Omega^2 e^{-\Omega^2/4}.
\end{equation}

Note that points in $B(R,\Omega,h,t)$ are at least a distance $\sqrt{2}\Omega$ from the origin and $B(R,\Omega,h,t)$ consists of parts of a sphere of radius $\Omega-h\leq\Omega$.  Thus we obtain
\begin{equation}\label{outthere2}
|B(R,\Omega, h,t)| = D\Omega^2 e^{-\Omega^2/4}.
\end{equation}

Assuming $\Omega\geq R$ we obtain putting together \eqref{outthere1} and \eqref{outthere2} for some $E>0$, 
\begin{equation}\label{ends}
|Ends(R,\Omega,h,t)|= E\Omega^2 e^{-\Omega^2/4}.
\end{equation}

We define a one-parameter family $\{\Sigma_3^R(t)\}_{t\in [0,1]}$ (with $\Sigma_3^R(0) = \Sigma_2^R(1)$) as follows:
\begin{equation}
\Sigma_3^R(t) = G(R,\Omega, h,t)\cup F_h(h) \cup\tilde{L}(R,h,h,h^3))\setminus (D^1(h,h,h^3)\cup \tilde{D}^2(R,h^3,1)). 
\end{equation}
Since by assumption $R\geq .2$ by Proposition \ref{conesbound}, \eqref{ends}, \eqref{taylorlittle} and \eqref{allcyl} we get 
\begin{equation}
|\Sigma_3^R(t)| < 2-\delta_2 + Ah^3+ Bh^2+hR+ E\Omega^2e^{-\Omega^2/4}
\end{equation}
Choose $\Omega_1$ so large so that
\begin{equation}
E\Omega_1^2e^{-\Omega_1^2/4}\leq \frac{\delta_2}{4}
\end{equation}
and $h_3$ small enough so that 
\begin{equation}
5h_3+ Ah_3^3+ Bh_3^2\leq \frac{\delta_2}{4}.
\end{equation}
Then we get for $h\leq\min(h_0,h_1,h_2, h_3,.1)$ and $\Omega\geq\max(\Omega_1, R)$:
\begin{equation}\label{first}
\sup_{t\in [0,1]} |\Sigma_3^R(t)| < 2- \frac{\delta_2}{2}.
\end{equation}

\noindent
{\it Step 4: Enlarging inner cylinder.}
Set $r_{max}(R) = \max{(R/2,R-1)}$.   For $t\in [0,1]$,  let us define the family of radii
\begin{equation}
r_t= h+t(r_{max}(R)-h). 
\end{equation}
For $t\in [0,1]$ let us consider the family of inner cylinders with expanding radii 
\begin{equation}
F_h(r_t) = D(r_t,h)\cup Cyl(r_t, h).
\end{equation}
Let us define a one-parameter family $\{\Sigma_4^R(t)\}_{t\in [0,1]}$ (with $\Sigma_4^R(0) = \Sigma_3^R(1)$) as follows:
\begin{equation}\label{inccyl}
\Sigma_4^R(t) = F_h(r_t)  \cup G(R, \Omega,h, 1) \cup \tilde{L}_i(R, h, r_t,h^3)\setminus (D^1(r_t,h,h^3)\cup \tilde{D}^2(R,  h^3,1)). 
\end{equation}
Thus we obtain
\begin{equation}\label{areas}
|\Sigma_4^R(t)| \leq |D(R, \infty, h)| +|Cyl(R,h)|+|F_h(r_t)| +Ah^3+ E\Omega^2 e^{-\Omega^2/4}
\end{equation}

\noindent
Recalling the formula for area of a cylinder and that $R\leq 5$

\begin{equation}\label{e}|Cyl(R,h)|\leq hRe^{-R^2/4}\leq 5h.\end{equation}
and
\begin{equation}
|Cyl(r_t,h)|\leq 5h, 
\end{equation}
\noindent
Using these we obtain the Gaussian area bound
\begin{equation}\label{areas2}
|\Sigma_4^R(t)| \leq |D(R, \infty, h)| +|D(0,r_{t},h)|+10h+ Ah^3+ E\Omega^2 e^{-\Omega^2/4}.
\end{equation}
\noindent

Applying this together with the formula for Gaussian area of disks, we get 
the Gaussian area bound
\begin{equation}\label{step4}
\sup_{t\in [0,1]} |\Sigma_4^R(t)|\leq e^{-h^2/4}(2-|D(r_{t},R)|)+10h+2Ah^3+E\Omega^2 e^{-\Omega^2/4}.
\end{equation}
We expand the first term in $h$ and get
\begin{equation}\label{final}
\sup_{t\in [0,1]} |\Sigma_4^R(t)|\leq 2-|D(r_{t},R)|+10h+Ah^3+ h^4/16+E\Omega^2 e^{-\Omega^2/4}.
\end{equation}
Note finally that
\begin{equation}
\inf_{R\in [.2,5], t\in [0,1]} |D(r_{t}(R),R)|\geq \eta_2>0.
\end{equation}
Choose $\Omega_2$ large enough so that 
\begin{equation}
E\Omega_2^2 e^{-\Omega_2^2/4}\leq \eta_2/4.
\end{equation}
Choose $h_4$ so that
\begin{equation}
Ah_4^3+ h_4^4/16 \leq \eta_2/4.
\end{equation}
Thus we get for all $h\leq\min(h_0,h_1,h_2,h_3,h_4,.1)$ and $\Omega\geq\max(\Omega_1,\Omega_2, R)$ the bound
\begin{equation}\label{finall}
\sup_{t\in[0,1]} |\Sigma_4^R(t)| < 2-\frac{\eta_2}{4}.
\end{equation}
\noindent
{\it Step 5: Catenoid estimate.}
Set \begin{equation} r_{necks}(R)= (R+r_{max}(R))/2.\end{equation}   Consider the set $W$ obtained as the intersection of the circle of radius $r_{necks}(R)$ about the origin in the plane $H$ with the rays $\cup_{i=2, even}^{2g+2} R_i$.   Let $V(R)$ denote the union of those lines that contain a point in $W$ and are orthogonal to $H$.   Let $V(R,\delta)$ denote the boundary of the the $\delta$-tubular neighborhood about the set of lines $V(R)$.  Let $D^3(h, \delta)$ denote the union of the disks in  $D(0,R+\Omega-h,h)$ bounded by the intersection $D(0,R+\Omega-h,h)\cap V(R,\delta)$.   Let $V(R,h, \delta)$ denote the $g+1$ components of $V(R,\delta)\setminus D(0,R+\Omega-h,h)$ intersecting the plane $H$ (i.e. the vertical tubes).  Note that for some $F>0$ there holds
\begin{equation}
|V(R,h,\delta)|\leq F\delta.
\end{equation}

By the catenoid estimate (Theorem 2.4 in \cite{KMN}) there is an isotopy $\{\Sigma_5^R(t)\}_{t\in [0,1]}$ so that $\Sigma_5^R(0) = \Sigma_4^R(1)$ and which ends\footnote{This isotopy is the time-reversed version of Theorem 2.4 in \cite{KMN}.} at 
\begin{equation}\label{finalfive}
\Sigma_5^R(1) = D(0,\Omega+R-h, h) \cup Ends(R,\Omega,h,1)\cup V(h, \delta) \setminus D^3(h, \delta).
\end{equation}
Roughly speaking, the surface $\Sigma_5(1)$ looks like two planes parallel to $H$ capped off with spherical-like surfaces joined together by thin vertical tubes.

Taking $\delta=h^3$ we estimate
\begin{equation}
|\Sigma_5^R(1)| = 2 - \frac{1}{2}h^2 + Fh^3+E\Omega^2e^{-\Omega^2/4}.
\end{equation}
In order to apply the catenoid estimate we further shrink $h\leq h_c= \inf_{R\in [.2,5]} h_c(R)$ where $h_c(R)$ is the required value of $h$ in Theorem 2.4 in \cite{KMN}, and then finally set $\Omega_3= \Omega_
3(h)$ so that for all $\Omega\geq\Omega_3$ there holds
\begin{equation}
E\Omega^2e^{-\Omega^2/4}\leq h^3.
\end{equation}
Choosing $h\leq\min(h_c,h_0,h_1,h_2,h_3,h_4,.1)$ and then $\Omega(R,h)\geq\max{(R, \Omega_1,\Omega_2,\Omega_3(h)})$\footnote{Note now that any further shrinking of $h$ requires adjusting $\Omega$ to be larger.}, we obtain
\begin{equation}
|\Sigma_5^R(1)| \leq 2 - \frac{1}{2}h^2 + Gh^3\leq 2-\frac{1}{4}h^2.
\end{equation}
The isotopy provided by the catenoid estimate satisfies 
\begin{equation}
\sup_{t\in [0,1]} |\Sigma_5^R(t)| = 2 - Ch^2,
\end{equation}
for some $0<C<1/4$ as desired.

The final surface in the isotopy $\Sigma^R_5(1)$ consists of two components $P_1(R),P_2(R)$ with \begin{equation}P_1(R)=(D(0,\Omega+R-h,h)\cup Ends(R, \Omega, h,1))\cap\{z>0\}\end{equation} above $H$ and \begin{equation}P_2(R)=(D(0,\Omega+R-h, h)\cup Ends(R, \Omega,h,1))\cap\{z<0\}\end{equation} below $H$, joined by thin vertical tubes $V(R,h^3,h)$ through $H$.    
\\
\\
\noindent
{\it Step 6: Reparameterization.}   We concatenate the previous five isotopies by defining for each $.2\leq R\leq 5$

$$
\Sigma_t^R = 
\begin{cases}
\Sigma^R_1(6t)    &\mbox{ for        } 0\leq t\leq 1/6 ,\\
\Sigma^R_2(6t-1)  & \mbox{ for        } 1/6\leq t\leq 2/6 ,\\
\Sigma^R_3(6t-2) &\mbox{ for        } 2/6\leq t\leq 3/6 ,\\
\Sigma^R_4(6t-3) &\mbox{ for        } 3/6\leq t\leq 4/6 ,\\
\Sigma^R_5(6t-4) &\mbox{ for        } 4/6\leq t\leq 5/6 ,\\
\end{cases}
$$


Let $\Psi$ denote a diffeomorphism from $[0,\infty)$ to $[0,1)$ and denote $s_1:=\Psi(.2)$ and  $s_2:=\Psi(5)$.   For $s_1\leq s\leq s_2$ let us define the family of surfaces  \begin{equation}\Sigma_{s,t}:= \Sigma^{\Psi^{-1}(s)}_t.\end{equation}

\begin{prop}\label{contfamily}
The family $\{\Sigma_{s,t}\}_{(s,t)\in [s_1,s_2]\times [0,5/6]}$ is an $\mathbb{P}_{g+1}$-invariant genus $g$ family of surfaces satisfying
\begin{equation}\label{gt}
\sup_{(s,t)\in [s_1, s_2]\times [0,5/6]} |\Sigma_{s,t}|<2.  
\end{equation}
\end{prop}

\subsection{Interpolation: Completion of proof of Theorem \ref{optimalsweepout}}
In this section, we complete the proof of Theorem \ref{optimalsweepout} by extending the sweepout $\Sigma_{s,t}$ defined on $[s_1,s_2]\times [0,\frac{5}{6}]$ to the entire rectangle $[0,1]\times[0,1]$ with the properties posited in Theorem \ref{optimalsweepout}.   In particular,  we must interpolate between the family of inversions given by the left and right boundaries of ($[s_1,t]_{t\in [0,5/6]}$, and $[s_2,t]_{t\in [0,5/6]}$) with the optimal sweepouts asserted at the left and right faces of the parameter space in Theorem \ref{optimalsweepout}.  Additionally, each surface on the bottom face $[s,5/6]_{s\in [s_1,s_2]}$ consists of two components as described above, one in $\{z>0\}$ and the other in $\{z<0\}$ joined by thin vertical necks and with total Gaussian area less than but very close to $2$.  We need to bring these areas down in order to have only small surfaces with small Gaussian area on the bottom face of our sweepout as posited in Theorem \ref{optimalsweepout}. 

The following lemma will be useful for this latter task.  It makes use of the mean-convex nature of half-spaces in the Gaussian metric.  

\begin{lemma}[Squeezing convex neighborhoods]\label{squeezinglemma}
For $k\geq 0$, let $\{\Gamma_t\}_{t\in S^k}$ be a smoothly varying family of rotationally symmetric two-spheres\footnote{For $k=0$, then $S^{0}$ denotes two points and the condition that the family is smoothly varying is vacuous.} contained $\{z>0\}$ such that
\begin{equation}
\sup_{t\in S^k} F(\Gamma_t) < \Lambda.  
\end{equation}
Then there exists a smoothlhy varying family of of rotationally symmetric two-spheres $\{\Gamma'_t\}_{t\in D^{k+1}}$ extending $\{\Gamma'_t\}_{t\in S^k}$ so that \begin{equation}\Gamma'_t\subset \{z>0\}\end{equation} for all $t\in D^{k+1}$ and 

\begin{equation}
\sup_{t\in D^{k+1}} F(\Gamma'_t) < \Lambda.  
\end{equation}
\end{lemma}
\begin{proof}
By Smale's theorem \cite{Smale},  there exists an extension $\{\tilde{\Gamma}_t\}_{t\in D^{k+1}}$ of $\{\Gamma_t\}_{t\in S^k}$.  Let us suppose without loss of generality that
\begin{equation}
\sup_{t\in S^1} F(\tilde{\Gamma}_t) < \Lambda'.
\end{equation}
with $\Lambda'>\Lambda$.  

For $\tau\in\mathbb{R}$
denote the translation $T_\tau:\mathbb{R}^3\to\mathbb{R}^3$
\begin{equation}
T_\tau(x,y,z) = (x,y,z+\tau).
\end{equation}
Let us parameterize the disk $D^{k+1}$ by $(\rho,\theta)$ with $0\leq \rho \leq 1$ and $\theta\in S^k$.  
Then for any $\tau_)>0$ 
Let us now construct the extension $\Gamma_t$ defined for $t\in D^{k+1}$.
For $1/2\leq \rho \leq 1$ let us define
\begin{equation}
\Gamma'_{\rho,\theta} = \Gamma_{\rho, \theta}-2(t-1)\tau_0.
\end{equation}
For $0\leq \rho\leq 1/2$ let us define
\begin{equation}
\Gamma'_{\rho,\theta} = \tilde{\Gamma}_{2(\rho-1/2), \theta}+\tau_0.
\end{equation}
In light of Lemma \ref{easy} we obtain
\begin{equation}
\sup_{t} F(\Gamma'_t)\leq\mbox{max}({\Lambda,  e^{-\rho_0^2/4}\sup_t F(\tilde{\Gamma}_t)})\leq \mbox{max}({\Lambda, e^{-\rho_0^2/4}\Lambda'}).
\end{equation}
Choosing $\rho_0$ so large so that
\begin{equation}
\Lambda' e^{-\rho_0^2/4}<\Lambda, 
\end{equation}
gives the result.  
\end{proof}
We now complete the proof of Theorem \ref{optimalsweepout} by providing the desired extension.  Note that we may need to shrink $h$ further and thus increase $\Omega(5,h)$ in the following sections.  
\\
\\
\noindent
\emph{Left side:}
On the left boundary $[s_2,t]_{t\in [0,5/6]}$,  we first taper down the isotopies in Step 3 and then the isotopy from Step 2.   Toward that end,  for $\frac{s_1}{2} \leq s\leq s_1$,  and $0\leq t\leq \frac{5}{6}$ let us define for $K=\Psi^{-1}(.2)$
\begin{equation}
\Sigma_{s,t} = 
\begin{cases}
\Sigma^K_1(6t)   &\mbox{ for            } 0\leq t\leq 1/6 ,\\
\Sigma^K_2(6t-1) &\mbox{ for        } 1/6\leq t\leq 2/6 ,\\
\Sigma^K_3(2(6t - 2)(s-s_1/2)/s_1) &\mbox{ for        } 2/6\leq t\leq 3/6 ,\\
\Sigma^K_4(6t-3) &\mbox{ for        } 3/6\leq t\leq 4/6 ,\\
\Sigma^K_5(6t-4) &\mbox{ for        } 4/6\leq t\leq 5/6 ,\\
\end{cases}
\end{equation}

In this way, for $(s,t)\in [\frac{s_1}{2},s_1]\times [2/6,3/6]$ the truncated cones in Step 3 only open partially instead of becoming entirely planar.  
Because $R\leq .2$,  it follows from item (2) in Proposition \ref{areacones} that the end result of Step 3 has less Gaussian area than $|\Sigma_{s_1,{3/6}}|$.   Since the supports of the isotopies in Step 4, and Step 5 are disjoint from the conical pieces where adjustments have now been made in Step 3, it follows that 
\begin{equation}
\sup_{s\in [\frac{s_1}{2},s_1], t\in [0,\frac{5}{6}]}|\Sigma_{s,t}|<2.
\end{equation}

For $(s,t) \in [\frac{s_1}{4},\frac{s_1}{2}]\times [0,\frac{5}{6}]$ we taper down the isotopy of Step 2 by defining
\begin{equation}
\Sigma_{s,t} = 
\begin{cases}
\Sigma^K_1(6t)   &\mbox{ for            } 0\leq t\leq 1/6 ,\\
\Sigma^K_2(4(6t-1)(s-s_1/4)/s_1) &\mbox{ for        } 1/6\leq t\leq 2/6 ,\\
\Sigma^K_3(0) &\mbox{ for        } 2/6\leq t\leq 3/6 ,\\
\Sigma^K_4(6t-3) &\mbox{ for        } 3/6\leq t\leq 4/6 ,\\
\Sigma^K_5(6t-4) &\mbox{ for        } 4/6\leq t\leq 5/6 ,\\
\end{cases}
\end{equation}

Since the extension $\{\Sigma_{s,t}\}_{(s,t)\in [s_1/4,s_1/2]\times [0,5/6]}$ merely shrinks the height of a capped cylinder of radius $.2$,  it is clear the Gaussian areas of the above family are also much less than $2$.

Then we fill in the rectangle spanned by $[s_1/8,s_1/4]\times [0, 5/6]$ by first tapering down the isotopy provided by Step 5 and then tapering down the isotopy of Step 4 so that $\{s_1/8\}\times [0,5/6]$ only consists of the isotopy provided in Step 1.  Then it is straightforward to extend to the rectangle $[0,s_1/8]\times [0,5/6]$ by adjusting Step 1 to shrink the necks and inner component $F_h(h)$ down to zero so that $\{\Sigma_{0,t}\}_{t\in [0,\frac{5}{6}]}$ gives the optimal foliation $\{S_{1-\tau}\}$ for $\tau$ in a suitable range, as desired.

Finally, let us apply  the $k=0$ case of Lemma \ref{squeezinglemma} to extend the sweepout $\Sigma_{s,t}$ defined at $(s_1/4,5/6)$ to the segment $\{(s_1/4,t)\}_{t\in [5/6,1]}$ so that $\Sigma_{s_1,1}=\Theta$ is a union of a tiny $\mathbb{Z}_\infty$-invariant round sphere $\chi\subset\{z>0\}$ centered on the positive $z$-axis together with $\tau_H(\chi)$ and $g+1$ very thin tubes joining the two components.  By Lemma \ref{squeezinglemma}, we guarantee the Gaussian areas of this extension are less than $2$.
\\
\\
\emph{Right side:}
On the right side $[s_2,t]_{t\in [0,5/6]}$, we need to first amend Step 4 in the construction of $\{\Sigma_{s_2,t}\}_{t\in [0,5/6]}$ where instead of enlarging the radius of the inner cylinder while keeping its height fixed,  we convert the inner component to ellipsoids, and then allow the heights (minor axes) of the ellipsoids to increase in tandem with their radii (major axis) so that at the end of the deformation we arrive at the spherical foliation. Let us give the details.  

For any $h,r\in (0,\infty)$ the ellipsoid $E(r,h)$ is inscribed inside the capped cylinder $F_h(r)$ defined in \eqref{defofsmall}.  Let $\{E(r,h,\lambda)\}_{\lambda\in [0,1]}$ denote a family of $\mathbb{P}_\infty$-invariant convex sets satisfying: 
\begin{enumerate}[i.]
\item $E(r,h,1) = F_h(r)$
\item $E(r,h,0) = E(r,h)$
\item $E(r,h,\lambda)\subset E(r,h,\lambda')$ if $\lambda'\geq \lambda$
\item For each $r$, if $h$ is small enough there holds for some $C_r>0$: \begin{equation}\label{el}|E(r,h,\lambda)|\leq 2|D(0,r,0)| + C_rh \mbox{ for all } \lambda\in [0,1].
\end{equation}
\end{enumerate}

For $(s,t)\in [s_2, s_2+\iota]\times [0,5/6]$, we amend Step 4 by replacing the term $F_h(r)$ in equation \eqref{inccyl} with the surface $E(r_t, h, -(s - (s_2+\iota))/\iota)$.  

Using \eqref{el} and \eqref{step4}, if $h$ is small enough, we get a similar expansion for the area along this amended isotopy that we denote $\tilde{\Sigma}^s_4(t)$, 
\begin{equation}
\sup_{t\in [0,1]} |\tilde{\Sigma}^s_4(t)|\leq e^{-h^2/4}(2-|D(r_{t},5,0)|)+(5+C)h+2Ah^3+E\Omega^2 e^{-\Omega^2/4}.
\end{equation}
Shrinking $h$ and then increasing $\Omega(5,h)$ as in equations \eqref{final}-\eqref{finall} we obtain since $r_t\leq 4$
\begin{equation}
\sup_{s\in [s_2,s_2+\iota]}\sup_{t\in [0,1]} |\tilde{\Sigma}^s_4(t)|\leq 2-\eta_2/4.
\end{equation}

For  $(s,t)\in [s_2+\iota,s_2 + 2\iota]\times [0,5/6]$ we amend Step 4 (and so adjust in $[s_2+\iota,s_2 + 2\iota]\times [3/6,4/6]$) by replacing the term $F_h(r)$ in equation \eqref{inccyl} with $E(r_t, h_{t,s})$ where $r_t$ is as defined in Step 4 and the height is given by
\begin{equation}
h_{t,s} = \gamma_1(t)+ \gamma_2(t)s, 
\end{equation}
where $\gamma_1(t) = (r_t - h)/\iota$ and $\gamma_2(t) = h - (r_t-h)(s_2+\iota)/\iota$.  The functions $\gamma_1(t)$ and $\gamma_2(t)$ are chosen so that for all $0\leq t\leq 1$
\begin{enumerate}
\item $h_{t,s_2+\iota}=h$
\item $h_{t,s_2+2\iota}=r_t$.
\end{enumerate}
Since $r_t\leq 4$ we get $|h_{t,s}|\leq 4$ and thus we have from \eqref{ellipsoidaway} and \eqref{step4} that 
\begin{equation}\label{areas22}
\sup_{t\in [0,1]}|\tilde{\Sigma}^s_4(t)| \leq 2-\delta_3 +2e^{-5^2/4} +5h+ Ah^3+ E\Omega(5)^2 e^{-\Omega(5)^2/4}.
\end{equation}
Since $\delta_3=.065>2e^{-5^2/4}\approx.0039$, shrinking $h$ sufficiently and increasing $\Omega(5,h)$ sufficiently guarantees the Gaussian area bound


\begin{equation}
\sup_{s\in [s_1+\iota,s_1+2\iota]}\sup_{t\in [0,1]}|\tilde{\Sigma}^s_4(t)|< 2.
\end{equation}


As a result of the modification of Step 4 for the range $s\in [s_2+\iota,s_2+2\iota]$, the ellipsoid part of $\Sigma_4^s(1)$ juts ``higher" than the portion at constant height $h$.  Thus we need to also amend the surfaces which are the end result of Step 5, which we denote $\tilde{\Sigma}^s_5(1)$.  The end result of Step 5 will no longer be as in \eqref{finalfive}  comprised of disks parallel to $H$ (plus caps) but instead parts of the graphs $z_{h,4,b(s)}$ defined in \eqref{combinedgraphs} (plus caps).  Recall the function $b(s):=h_{1,s}$ interpolates from $b(s_2+\iota)=h$  to $b(s_2+2\iota)=4$. We set 
\begin{equation}
\tilde{\Sigma}^s_5(1) = (z_{h,4,b(s)}\cap B_{\Omega+R-h}(0)) \cup Ends(R,\Omega,h,1)\cup V(R,h, \delta) \setminus D^3(R,h, \delta).
\end{equation}

By Lemma \ref{graphs}, possibly shrinking $h$ further and increasing $\Omega(5,h)$ we obtain for some $c>0$
\begin{equation}
\sup_{s\in [s_2+\iota, s_2+2\iota]}\sup_{t\in [0,1]} |\tilde{\Sigma}^s_5(t)|\leq 2-ch^2.\end{equation}


This completes the extension to the rectangle spanned by $(s,t)\in [s_2,s_2+2\iota]\times [0,5/6]$.  We extend to $(s,t)\in [s_2+2\iota,s_2+3\iota]\times [0,5/6]$ by first tapering down Step 5 and then and then tapering down Steps 3 and 2. It is then straightforward to shrink the tubes from Step 1 down to zero.  In this way, $\Sigma_{t,0}$ coincides with the optimal foliation $S_\tau$ for $\tau$ in a suitable range of $[0,1]$.

Finally, on the arc $\{s_2+2\iota,t\}_{t\in [\frac{5}{6},1]}$ we again use Lemma \ref{squeezinglemma} to isotope the surface $\Sigma_{s_2+2\iota, 5/6}$ to $\Theta$ with Gaussian areas less than $2$.  One can check that the volume bounded by the surface $\Sigma_{s_2+2\iota, 5/6}$, and thus also surfaces assigned to this arc, all bound a volume much less than $V/2$ on one side as required in Proposition \ref{optimalsweepout}.
\\
\\
\emph{Bottom:}
On the arc $\{s,1\}_{s\in [s_1/4,s_2+2\iota]}$ we set $\Sigma_{s,1}:= \Theta$.  We now  extend our sweepout $\Sigma_{s,t}$ to the solid rectangle $R\subset I^2$ with boundary $\Gamma$ given by the segments $\{s,\frac{5}{6}\}_{s\in [s_1/4,s_2+2\iota]}$,  $\{s,1\}_{s\in [s_1/4,s_2+2\iota]}$,  $\{s_1/4,t\}_{t\in [\frac{5}{6},1]}$, and $\{s_2+2\iota,t\}_{t\in [\frac{5}{6},1]}$.  By construction, along $\Gamma$, the surfaces $\Sigma_\Gamma$ consist of a $\mathbb{Z}_\infty$-invariant sphere in $\{z>0\}$ together with the mirror image of the sphere in $\{z<0\}$ and $g+1$ very thin tubes joining them.  After shrinking these tubes further and applying Lemma \ref{squeezinglemma} we obtain the desired extension to the rectangle $R$ in $I^2$ with $\partial R = \Gamma$.  Since
\begin{equation}
\sup_{(s,t)\in\partial R} |\Sigma_{s,t}|< 2,
\end{equation}
Lemma \ref{squeezinglemma} guarantees that the extension also satisfies
\begin{equation}
\sup_{(s,t)\in R} |\Sigma_{s,t}|< 2.
\end{equation}
In summary, we have extended the sweepout to the entire square $I^2$ except for the tiny corner rectangles $[0,s_1/4] \times [5/6,1]$ and $[s_2+2\iota,1]\times [5/6,1]$.  Reparameterizing the domain to be a square completes the proof of Theorem \ref{optimalsweepout}.


\section{Lusternick-Schnirelman argument}\label{sectionls}

Using the ``flipping" sweepout constructed in Theorem \ref{optimalsweepout} we obtain the following result (cf. Theorem 1.6 in \cite{Ket3}) .  Recall that there exists $\varepsilon_{BW}>0$ with the property that any self-shrinker that is not the plane, sphere or cylinder has entropy at least $\lambda(S^1_*\times\mathbb{R})+\varepsilon_{BW}$.

\begin{thm}[Flipping optimal foliation]\label{existencefinal}
For each $g\geq 1$ there exists a $\mathbb{P}_{g+1}$-invariant self-shrinker $\Sigma_g$ with
\begin{enumerate}
\item $\lambda(\mathbb{S}^1_*\times\mathbb{R})+\varepsilon_{BW}< \lambda(\Sigma_g) <2$\label{entropy}
\item $genus(\Sigma_g)\leq g$.
\end{enumerate}
\end{thm}


\begin{proof}
Let $\omega_2(\mathbb{P}_{g+1})$ denote the min-max width associated to the $\mathbb{P}_{g+1}$-equivariant saturation of $\{\Sigma_{s,t}\}_{(s,t)\in I^2}$ constructed in Theorem \ref{optimalsweepout}.   If  \begin{equation}\omega_2(\mathbb{P}_{g+1})>\sup_{(s,t)\in \partial I^2}|\Sigma_{s,t}|=\lambda(\mathbb{S}^2_*)\end{equation} then applying Theorem \ref{highparamminmax} we obtain a $\mathbb{P}_{g+1}$-invariant self-shrinker $\Sigma_g$ with entropy equal to $\omega_2(\mathbb{P}_{g+1})$. In light of the entropy bound $\omega_2(\mathbb{P}_{g+1})<2$ the self-shrinker $\Sigma_g$ occurs with multiplicity $1$.  By the genus bounds in Theorem \ref{highparamminmax} we obtain $genus(\Sigma_g)\leq g$.   By \cite{KL},  we get $\Sigma_g\neq \mathbb{S}^1_*\times\mathbb{R}$ as a two-parameter min-max procedure cannot produce the the equivariant index $1$ cylinder (Proposition \ref{indexbounds}). 

Thus it remains to rule out the equality case  \begin{equation}\omega_2(\mathbb{P}_{g+1})=\lambda(\mathbb{S}^2_*).\end{equation}  Let $\{\Sigma^i_{s,t}\}_{(s,t)\in I^2}$ denote a sequence of pulled-tight $\mathbb{P}_{g+1}$-sweepouts in the saturation satisfying
\begin{equation}\label{tight}
\sup_{(s,t)\in I^2} |\Sigma^i_{s,t}|\leq \lambda(S^2) + \frac{1}{i}.  
\end{equation}

For each $i>0$ and $\varepsilon>0$ let 

\begin{equation}\mathcal{S}^i_\varepsilon := \{(a,b)\in[0,1]\times[0,1] \;  | \;\mathcal{F}(\Gamma^i_{a,b}, \mathbb{S}^2_*)< \varepsilon\}\end{equation} 
\noindent
Note that for each $\varepsilon>0$ and positive integer $i$ we have, letting $t_0\in [0,1]$ be the unique value so that $S_{t_0}=\mathbb{S}^2_*$, that both $(0,t_0), (1,t_0)\in \mathcal{S}^i_\varepsilon$.

First we claim that for each $\varepsilon>0$,  there exists an integer $I(\varepsilon)$ large enough so that if $i>I(\varepsilon)$ then $\mathcal{S}^i_\varepsilon$ contains a continuous path $(a^i_\varepsilon(\eta), b^i_\varepsilon(\eta))_{\eta\in [0,1]}\subset [0,1]\times[0,1]$ beginning on the left side of the square and ending on the right side of the square.   Suppose not.   Then there exists $\varepsilon_0$ so that the claim fails.    Since the claim fails,  it follows that we can find a path $(c^i_\varepsilon(\tau) ,d^i_\varepsilon(\tau))_{\tau\in [0,1]}$ such that $(c^i_\varepsilon(0),d^i_\varepsilon(0))$ is on the bottom face of the square,  and $(c^i_\varepsilon(1),d^i_\varepsilon(1))$ is on the top face of the square so that
\begin{equation}\label{far}
F(\Phi^i_{c^i_\varepsilon(\tau),d^i_\varepsilon(\tau)},\mathcal{S}^i_\varepsilon) \geq \varepsilon_0 \mbox{ for all } \tau,
\end{equation}
for some subsequence of $i$ (not relabelled).

For each $i$, we consider the family $\Sigma^i_t:=\{\Phi^i_{c^i(\tau),d^i(\tau)}\}_{\tau\in [0,1]}$.  Let $\Pi_0$ denote the $\mathbb{P}_{g+1}$-saturation of the families $\{\Sigma^i_t\}_{t\in [0,1]}$.  Let us define the corresponding min-max width
\begin{equation}
\omega_1=\inf_{\Sigma'_t\in\Pi_0} \sup_{t\in [0,1]} |\Sigma'_t|.
\end{equation}

We claim 
\begin{equation} 
\omega_1=\lambda(\mathbb{S}^2_*).
\end{equation}

Indeed, let $\{\Sigma'_t\}_{t\in [0,1]}\in\Pi_0$.  Then we may consider a continuously varying family $\Omega_t$ of open sets so that $\Sigma_t=\partial\Omega_t$ for each $0\leq t\leq 1$.  Letting $g(t) = \mbox{vol}(\Omega_t)$ we have by item \ref{great} in Proposition \ref{optimalsweepout} that $g(0)< \eta_2<V/2$ and $g(1)>V-\eta_2>V/2$.  Thus for some $t_0\in (0,1)$ there holds $g(t_0)=V/2$.  By the Isoperimetric Inequality in Gaussian space \cite{ST}, this implies $|\Sigma'_{t_0}|\geq 1$.  Thus \begin{equation}\omega_1\geq 1.\end{equation}  On the other hand, by \eqref{tight} we have \begin{equation}\label{les} \omega_1\leq\lambda(\mathbb{S}^2_*)<2.\end{equation} 
Applying the Min-max Theorem \ref{highparamminmax} to the family $\Pi_0$, we obtain an embedded self-shrinker $\Gamma$ with entropy equal to $\omega_1$.  By item (2) in Theorem \ref{highparamminmax}, $\Gamma\neq H$ as $\Gamma$ must intersect the singular set of the group action orthogonally, but $H$ contains segments of the singular set.  By Bernstein-Wang \cite{BW} and \eqref{les} it follows that $\Gamma=\mathbb{S}^2_*$ and $\omega_1=\lambda(\mathbb{S}^2_*)$.

Thus the sequence of one-parameter families $\{\Phi^i_{c^i(\tau),d^i(\tau)}\}_{\tau\in [0,1]}$ is a minimizing sequence of sweepouts in $\Pi_0$.  On the other hand, by \eqref{far}, no min-max sequence obtained from it can be almost minimizing in annuli.  Thus by Pitts' combinatorial deformation (cf. Proposition 5.3 in \cite{CD}),  we obtain another element in $\Pi_0$ with maximal Gaussian area less than $\lambda(\mathbb{S}^2_*)$, implying
\begin{equation}
\omega_1 < \lambda(\mathbb{S}^2_*),
\end{equation}
a contradiction.  Thus the claim is established.

For each $\delta>0$ there exists $\varepsilon(\delta)>0$ so that the paths $(a^i_{\varepsilon(\delta)}(\eta), b^i_{\varepsilon(\delta)}(\eta))_{\eta\in [0,1]}$ joining the left side of the square to the right, concatenated with the paths connecting $(a^i_{\varepsilon(\delta)}(0),b^i_{\varepsilon(\delta)}(0))$ to $(0,t_0)$ and $(a^i_{\varepsilon(\delta)}(1),b^i_{\varepsilon(\delta)}(1))$ to $(1,t_0)$ on the left and right side, respectively,  is contained in $\mathcal{S}_\delta^i$.   This follows by contradiction because there by items \ref{a}-\ref{ee} in Proposition \ref{optimalsweepout} there is a unique point on the left side of the square (as well as on right) whose corresponding surface is a self-shrinking sphere and has Gaussian area $\lambda(\mathbb{S}^2_*)$.  Let us denote these concatenated paths by $(\tilde{a}^i_{\varepsilon(\delta)}(\eta), \tilde{b}^i_{\varepsilon(\delta)}(\eta))_{\eta\in [0,1]}$.  

Choosing $\delta$ small enough,  the paths $(\tilde{a}^i_{\varepsilon(\delta)}(\eta), \tilde{b}^i_{\varepsilon(\delta)}(\eta))_{\eta\in [0,1]}$ (whose corresponding surfaces are contained in a $\delta$-neighborhood about $\mathbb{S}^2_*$) join $\mathbb{S}^2_*$ to itself but with opposite orientation by item \ref{opposite} in Theorem \ref{optimalsweepout}.  This is impossible (cf. Lemma 3.2 in \cite{Ket3}). 

Thus we have ruled out the case that $\omega_2(\mathbb{P}_{g+1})=\lambda(\mathbb{S}^2_*)$.  This completes the proof.

\end{proof}

\section{Genus of $\Sigma_g$}\label{genussection}
In this section we prove Theorem \ref{main}\ref{3} which we restate: 

\begin{prop}\label{genustheorem}
For each $g\geq 1$, the self-shrinker $\Sigma_g$ has genus $g$.
\end{prop}

Min-max limits such as $\Sigma_g$ are obtained topologically by performing a sequence of neckpinches on the approximating sequence (\cite{K} and Section 5 of \cite{KetG}). To prove Proposition \ref{genustheorem} we must classify such compressions up to isotopy. In general, a (Heegaard) surface with genus at least $2$ has infinitely many such compressing disks which makes this a challenging task.  However,  because a fundamental piece of any surface in the equivariant saturation $\Pi_{\Sigma_{s,t}}$ of the family $\{\Sigma_{s,t}\}_{(s,t)\in I^2}$ constructed in Theorem \ref{optimalsweepout} is a sphere with prescribed orthogonal intersection with the singular set of the group action $\mathbb{P}_{g+1}$ on $\mathbb{R}^3$, one can classify the finitely many possible equivariant compressions. 

Roughly speaking, the necks joining the two parallel spheres in a surface $\Sigma_{s,t}$ from a sweepout contained in $\Pi_{\Sigma_{s,t}}$ can break one way to give two concentric spheres, or the other way to give a sphere contained in $\{z>0\}$ and another sphere in $\{z<0\}$.  A key point is that no compression can produce a torus.  According to Berchenko-Kogan's numerical analysis \cite{BKogan},  the Angenent torus likely has equivariant index $2$ and thus would otherwise be difficult to rule out as min-max surface $\Sigma_g$ obtained from a two-parameter min-max process\footnote{In fact, if we worked with the Almgren-Pitts version of min-max theory as opposed to the Simon-Smith one, this could well be the min-max limit for each $g$.}.

Let $G\subset O(3)$ act on $\mathbb{R}^3$.  For $x\in \mathbb{R}^3$ the \emph{isotropy subgroup} $G_x\subset G$ is
 \begin{equation}
G_x = \{g\in G\; |\; gx = x\}.  
 \end{equation}
The \emph{singular set} of the group action is 
\begin{equation}
S_G = \{x\in\mathbb{R}^3\;|\; G_x\neq e\}.  
\end{equation}
In general, the set $S_G$ admits a decomposition as
\begin{equation}
S_G=S_G^0\cup S_G^1\cup S_G^2
\end{equation} where each component of $S_G^2$ is an open subset of a plane with isotropy $\mathbb{Z}_2$ (``reflections"), each component of $S_G^1$ is a straight line segment (with isotropy containing a subgroup of rotations isomorphic to $\mathbb{Z}_n$ for some $n$) and $S_G^0\subset\{ (0,0,0)\}$.  

Assume $\Sigma$ is a $G$-equivariant surface that is orthogonal to any component of $S_G$ that it intersects.  Let $S\subset S_G^1$ denote a connected segment with $\mathbb{Z}_n\subset G_S$ where $\mathbb{Z}_n$ is a  subgroup of rotations and $G_S$ is the isotropy group along $S$.  We say \emph{$\Sigma'$ is obtained from $\Sigma$ by an $\mathbb{Z}_n$-neckpinch along $S$} if the following is true:  
\begin{enumerate}
\item There exists two embedded $G_S$-equivariant disks $D_1, D_2\subset\mathbb{R}^3$ which each intersects $S$ once and a $G_S$-equivariant annulus $A\subset \Sigma$, disjoint from $S$ with $\partial A = \partial D_1\cup \partial D_2$.
\item $D_1\cup D_2\cup A$ bounds a solid three ball $B$ with $B\cap (S_G^1\cup S_G^0)=B\cap S$.
\item Letting $G'$ denote the stabilizer of $B$ (as a set), \begin{equation}\Sigma'=(\Sigma\setminus \bigcup_{[g]\in G/G'} g(A)) \cup (\bigcup_{[g]\in G/G'} g(D_1)\cup\bigcup_{[g]\in G/G'}g(D_2).\end{equation}
\end{enumerate}

We say \emph{$\Sigma'$ is obtained from $\Sigma$ by an ordinary neckpinch} if the following is true: 
\begin{enumerate}
\item There exists two embedded disks $D_1, D_2\subset\mathbb{R}^3$ and an annulus $A\subset \Sigma$ all disjoint from $S_G^1$ such that $\partial A = \partial D_1\cup \partial D_2$.
\item $D_1\cup D_2\cup A$ bounds a solid three ball $B$ with $B\cap S_G\subset B\cap S_G^2$.
\item Letting $G'$ denote the stabilizer of $B$ as a set, \begin{equation}\Sigma'=(\Sigma\setminus\bigcup_{[g]\in G/G'} g(A)) \cup (\bigcup_{[g]\in G/G'} g(D_1)\cup\bigcup_{[g]\in G/G'}g(D_2).\end{equation}

\end{enumerate}

\begin{rmk}
Note that in items (2) we allow the neckpinch to occur through planes of reflection which is why we only demand $B$ to intersect the $1$-dimensional stratum of the singular set in a particular way.
\end{rmk}

To classify possible compressions it will suffice to consider the subgroup $\mathbb{D}_{g+1}\subset \mathbb{P}_{g+1}$ consisting of orientation-preserving isometries. Let $\Gamma_g$ denote a sweepout surface isotopic to $\Sigma_{s,t}$ for any $(s,t)\in\mbox{int}(I^2)$ as constructed in Proposition \ref{optimalsweepout}.   Let us consider the quotient $\overline{\Gamma}_g :=\Gamma_g/\mathbb{D}_{g+1}$ inside the orbifold $\overline{\mathbb{R}}^3:=\mathbb{R}^3/\mathbb{D}_{g+1}$.  Let $\pi:\mathbb{R}^3\to\overline{\mathbb{R}}^3$.  And set $\overline{R}_i^+=\pi(R_i)$ for each $i\in\{1,...,2g+2\}$ and $\overline{Z}^+=\pi(Z^+)$.  Note that for $i$ odd, we get $\overline{R}_1^+=\pi(R_i)$ and for $i$ even we get $\overline{R}_2^+=\pi(R_i)$.
In the orbifold $\overline{\mathbb{R}}^3$, the three rays $\overline{R}_1^+$,$\overline{R}_2^+$, and $\overline{Z}^+$ meet at the origin point $O$.

The projected surface $\overline{\Gamma}_g$ is a sphere intersecting $\overline{R}_1^+$ twice, $\overline{R}_2^+$ zero times, and $\overline{Z}^+$ twice.   For \emph{any} sphere $\overline{G}\subset\overline{R}^3$ let us define a triple of non-negative integers (``the intersection data") $\vec{g} = (k_1,k_2, b)$ denoting its number of intersection points with $\overline{R}_1^+$, $\overline{R}_2^+$ and $\overline{Z}^+$ respectively.  For example, $\overline{\Gamma}_g$ has intersection data $(2,0,2)$.  

Assuming $\pi^{-1}(\overline{G})$ is connected, the Riemann-Hurwicz formula recovers the genus of $\pi^{-1}(\overline{G})$ if $\overline{G}$ has intersection data $(k_1,k_2,b)$. In the following let $k=k_1+k_2$. 
\begin{equation}\label{rh}
genus(\pi^{-1}(\overline{G})) = \frac{k}{2}-\frac{1}{2}g(4-2b-k)-1.
\end{equation}

We have the following:

\begin{lemma}\label{surgprop} Let $(\overline{G},\vec{g})$ be a sphere embedded in $\overline{\mathbb{R}}^3$.  An equivariant neckpinch disconnects $\overline{G}$ into two spheres $(\overline{G}_1,\vec{g}_1)$ and $(\overline{G}_2,\vec{g}_2)$ with $\vec{g}_i=(n_i,m_i,a_i)$ so that the following are true.
\begin{enumerate}[i.]
\item For a $\mathbb{Z}_k$-neckpinch  we have $\vec{g}_1+\vec{g}_2-2\vec{v} = \vec{g}$ and $\vec{v}$ is either $(0,0,1)$, $(0,1,0)$ or $(1,0,0)$.  \label{i}
\item For an ordinary neckpinch there holds $\vec{g}_1+\vec{g}_2=\vec{g}$  \label{ii}
\item In both cases, for each $i=1,2$ the parity of $n_i$, $m_i$ and $a_i$ coincide.  \label{iii}
\end{enumerate}
\end{lemma}
\begin{proof}
To verify \ref{i} and \ref{ii},  observe that any $\mathbb{Z}_k$-neckpinch performed on $\overline{G}$ partitions the intersection vector $\vec{g}$ into two summands and adds one to the same entry of both summands while an ordinary neckpinch gives a partition of the intersection vector without adding one to any entry.

To show \ref{iii}, observe first that $\overline{\mathbb{R}}^3$, as a quotient of $\mathbb{R}^3$ by orientation-preserving isometries, is itself homeomorphic to $\mathbb{R}^3$ and thus the following topological considerations apply. A sphere contained in $\overline{\mathbb{R}}^3\setminus O$ and a ray with tip at $O$ have odd intersection number if the sphere bounds a region in $\overline{\mathbb{R}}^3$ containing $O$ and even intersection number if not.  
Each $\overline{G}_i\subset \overline{\mathbb{R}}^3\setminus O$ bounds a $3$-ball in $\overline{\mathbb{R}}^3$ which either contains the origin $O$ or not.  If it contains the origin, it has odd intersection number with each of the rays $\overline{R}_1^+$, $\overline{R}_2^+$, $\overline{Z}^+$ emanating from it.  Otherwise it has even intersection number with all three rays. This gives item \ref{iii}.

\end{proof}

The following is the main consequence of this analysis that we will use.  

\begin{prop}[Surgeries on Sweepout Surfaces]\label{surgeries}
Let $\Gamma\subset\mathbb{R}^3$ be a closed embedded $\mathbb{P}_{g+1}$-invariant genus $g$ surface with $\pi(\Gamma)\subset\overline{\mathbb{R}}^3$ a sphere with intersection type $(2,0,2)$.

Suppose $\Gamma$ is $\mathbb{P}_{g+1}$-equivariant and  is obtained from $\Phi$ after a sequence of $\mathbb{P}_{g+1}$-equivariant neckpinches.  Then $\Gamma_g$ consists of a surface of genus $g$ isotopic to $\Gamma$ together with a union of spheres, or else is a union of spheres.  

\end{prop}
\begin{proof}
The surface $\Phi$ has intersection data given by $(2,0,2)$.  We claim that performing an equivariant neckpinch on $\Phi$ either produces a union of spheres or else a surface of type $(2,0,2)$ and a union of spheres.  

Up to permuting the two components, there are five possible partitions of the intersection data of $\Phi$: $(1,0,1), (1,0,1)$ and $(1,0,2), (1,0,0)$ as well as $(2,0,0), (0,0,2)$ and $(2,0,1), (0,0,1)$, and finally $(2,0,2),(0,0,0)$.

First let us consider ordinary neckpinches, which by Lemma \ref{surgprop}ii are enumerated by the possible partitions above.   All of them are excluded by Lemma \ref{surgprop}\ref{iii} except for $(2,0,0), (0,0,2)$ which lifts to a union of spheres and $(2,0,2)$ and $(0,0,0)$ which corresponds to a genus $g$ surface together with a union of spheres.

To consider equivariant $\mathbb{Z}_k$-neckpinches, in each of the above five partitions, we have three possibilities for $\vec{v}$, leaving $15$ cases to consider.  We will show that all but those claimed in the statement of the proposition are ruled out by items \ref{iii} in Lemma \ref{surgprop}.   For the partition $(1,0,1), (1,0,1)$, only $\vec{v}= (0,1,0)$ does not violate \ref{iii}), giving resulting spheres with data $(1,1,1), (1,1,1)$.   By \eqref{rh} these spheres lift to a union of two-spheres.  For the partition, $(1,0,2), (1,0,0)$,  only $\vec{v} = (1,0,0)$ giving the resulting spheres with data $(2,0,2), (2,0,0)$ does not violate \ref{iii}.  For the partition $(2,0,0), (0,0,2)$, all choices of $\vec{v}$ violate item \ref{iii}.  For the partition  $(2,0,1)$, $(0,0,1)$, only setting $\vec{v}= (0,0,1)$ does not violate \ref{iii}, giving two spheres with data $(2,0,2), (0,0,2)$ for which the genera are again $g$ and $0$ by \eqref{rh}.
Finally for the partition $(2,0,2)$, $(0,0,0)$, no choice of $\vec{v}$ gives an admissible choice of $\vec{g}_1$ and $\vec{g}_2$.
\end{proof}
\noindent

Recall that for a smooth embedded surface $\Sigma\subset\mathbb{R}^3$ (possibly with boundary) after choosing a normal $n(p)$ on $\Sigma$ we denote the tubular neighborhood:
\begin{equation}
T_\epsilon(\Sigma) = \{\exp_{p}(tn(p))\;|\; p\in\Sigma, t\in [-\epsilon,\epsilon]\}.
\end{equation}
For $\epsilon$ small enough, $T_\epsilon(\Sigma)$ is diffeomorphic to $\Sigma\times [-\epsilon,\epsilon]$.  Moreover, $\partial T_\epsilon(\Sigma)$ consists of two components $\partial T^+_\epsilon(\Sigma)$ and $\partial T^-_\epsilon(\Sigma)$ each isotopic to $\Sigma$ given by 
\begin{equation}
\partial T^\pm_\epsilon(\Sigma)= \{\exp_{p}(tn(p))\;|\; p\in\Sigma, t=\pm\epsilon\}.
\end{equation}
\noindent
The families $\{\partial T^\pm_\delta(\Sigma)\}_{\delta\in [0,\epsilon]}$ smoothly foliate a neighborhood of $\Sigma$.  

For some $\epsilon_0$ small enough, by virtue of this smooth foliation, the following is true.  There exist $C>0$ and $\lambda>0$ so that for any closed curve $\gamma$ contained in $\partial T^\pm_\epsilon(\Sigma)$ for some $\epsilon\in [0,\epsilon_0]$ of length at most $\lambda$, there exists an embedded disk $D\subset\partial T^\pm_\epsilon(\Sigma)$ with boundary $\gamma$ and diameter at most $C\lambda$ (cf. Proposition 2.3 in \cite{DP}).
\\
\newline
\emph{Proof of Proposition \ref{genustheorem}.}
Let us show that \begin{equation}genus(\Sigma_g) = g.\end{equation}  By the Gaussian area bound $\lambda(\Sigma_g)>\lambda(S^1\times\mathbb{R})$ (item 2 in Theorem \ref{existencefinal}) and Brendle's classification \cite{B} we get that $genus(\Sigma_g)>0$.  Thus we may assume toward a contradiction that \begin{equation}\label{isbelow} genus(\Sigma_g)=h\mbox{      with     } 0 <h<g.\end{equation} For large large $i$, after performing finitely $\mathbb{Z}_k$ and ordinary neckpinches on the min-max sequence $\Gamma_i$ (and discarding some connected components of the result) we obtain a surface $\overline{\Gamma}_i$ contained in a tubular neighborhood about $\Sigma_g$ and isotopic to $\Sigma_g$.   By Proposition \ref{surgeries}, the genus of $\Gamma'_i$ is equal to $0$ or $g$.  Since $genus(\overline{\Gamma}'_i)=genus(\Sigma_g)$ we get a contradiction to \eqref{isbelow}.  Let us give the details. 

The self-shrinker $\Sigma_g$ intersects the $1$-dimensional part of the singular set of the group action $S_G^1$ in finitely many points $p_1,...,p_k$.  Since $\partial T_\lambda(\Sigma_g)$ consists of two parallel components it thus intersects $S_G^1$ in $2k$ points which we denote $\{p^\pm_1(\lambda),...,p^\pm_k(\lambda)\}$.

For any $\delta>0$, by the varifold convergence $\Gamma_i\to \Sigma_g$, we obtain for $i$ large enough 
\begin{equation}
\mathcal{H}^2(\Gamma_i\cap (T_\epsilon(\Sigma_g)\setminus T_{\epsilon/2}(\Sigma_g))\leq \delta.  
\end{equation}
Fix $\epsilon<\epsilon_0$.  By the co-area formula and Sard's lemma, we may choose $\eta_i\in [\epsilon/2,\epsilon]$ so that
\begin{equation}\label{smallcircles}
\mathcal{H}^1(\Gamma_i\cap \partial T_{\eta_i}(\Sigma_g))<4\delta/\epsilon.
\end{equation}
and so that $\Gamma_i$ intersects $\partial T_{\eta_i}(\Sigma_g)$ transversally in a union of circles $C^i_1,...,C^i_k$.  If $\delta$ is sufficiently small by \eqref{smallcircles} and the choice of $\epsilon$, each of these circles has short length and bounds a small disk of bounded diameter in $\partial T_{\eta_i}(\Sigma_g)$ which thus can only intersect at most one of the points $\{p^\pm_1(\eta_i),...,p^\pm_k(\eta_i)\}$. 

We may consider the corresponding circles in $\Gamma_i\cap \partial T_{\eta_i\pm \tau}(\Sigma_g)$ for some tiny $\tau$.  By transversality, these circles bound annuli and we may perform a series of neckpinch surgeries supported in $T_{\eta_i+\tau}(\Sigma_g)\setminus T_{\eta_i-\tau}(\Sigma_g)$. If the circles $C^i_l$ along which we surger contain one of the points $p_l(\eta_i\pm\tau)$ in their interior, then one performs a $\mathbb{Z}_k$-neckpinch. Otherwise, it is an ordinary neckpinch.  To arrive at $\Gamma'_i$  after this surgery process we then discard any connected components contained in $\mathbb{R}^3\setminus T_{\eta_i-\tau}(\Sigma_g)$.  

The resulting surface $\Gamma'_i$ has the following properties: 
\begin{enumerate}
    \item $\Gamma'_i=\Gamma_i$ in $T_{\eta_i-2\tau}(\Sigma_g)$
    \item $\Gamma_i'\subset T_{\eta_i}(\Sigma_g)$.
    \item $genus(\Gamma_i')= g\mbox{ or } 0$.
\end{enumerate}


It follows from Section 5.1 in \cite{K} that using the improved version of Simon's lifting lemma one may perform further equivariant neckpinches and further isotopies on $\Gamma_i$ within $T_{\eta_i}(\Sigma_g)$ in order to obtain a surface $\overline{\Gamma}_i$ isotopic to $\Sigma_g$. This completes the proof.

Finally let us give the straightforward modifications in the case that $\Sigma$ is noncompact.  Choose $R$ so large so that for $s\geq R$, the set $B_{s}(0)\cap\Sigma_g$ consists of curves $E_1(s),...E_k(s)$ that each bounds an end of $\Sigma_g$ diffeomorphic to $\mathbb{S}^1\times[0,\infty)$.  Replace in the previous argument $\Gamma_i$ with $\Gamma_i\cap B_{R+1}(0)$ and $\Sigma_g$ with $\Sigma_g\cap B_{R+1}(0)$.  Then any curves with some support in $\partial T_{\eta_i}(\Sigma_g)\cap B_{R}(0)$ are closed and avoid $\partial B_{R+1}(0)$ by the choice of $\epsilon_0$ above (otherwise, their diameter would be too large).  Thus in the same way we may perform equivariant neckpinches on $\Gamma_i$ to obtain a surface $\Gamma_i'$ with $\Gamma'_i\cap B_R(0)\subset T_{\eta_i}(\Sigma_g)$.  By the improved version of Simon's lifting lemma, after isotopies and further surgeries on $\Gamma_i'\cap B_R(0)$, we obtain a surface $\Gamma_i''$ isotopic to $\Sigma_g\cap B_R(0)$.  
\qed




\section{Limit of $\Sigma_g$}\label{limitsection}

In this section we prove Theorem \ref{main}\ref{4} and \ref{6} which we restate: 

\begin{thm}\label{mainsec}
The self-shrinkers $\Sigma_g$ satisfy the following properties:  
\begin{enumerate}[(i)]
\item 
 \begin{equation}\label{limit}\lim_{g\rightarrow\infty}\Sigma_g= 2H \mbox{ as varifolds.}\end{equation}
\item For any subsequence $g\to\infty$, up to taking a further subsequence, the convergence in \eqref{limit} is smooth on compact subsets away from a single circle.
\end{enumerate}
\end{thm}

Let $\Sigma_\infty$ be a varifold limit of $\Sigma_g$ for some subsequence $I$ of $g$.  We will show $\Sigma_\infty = 2H$, which implies \eqref{limit}. The varifold $\Sigma_\infty$ is an $F$-stationary integral varifold invariant under the group $\mathbb{P}_\infty$. Let us define the circle \begin{equation}C(p,r) = \{(x,y,z)\; |\; x^2+y^2 = r^2 \mbox{ and } z = p\} \end{equation}
\noindent
Let $N(p,r,\epsilon)$ denote the $\epsilon$-tubular neighborhood about $C(p,r)$.

We need the following definition (letting $K$ denote any subsequence of $I$):
\begin{equation}
\begin{split}
\mathcal{S}(K) :=\{C(p,r) \; |\; & \mbox{ for some subsequence $J\subset K$ and all } \epsilon>0, \mbox{ the convergence} \\ & \; \Sigma_{J}\mres N(p,r,\epsilon)\to \Sigma_\infty\mres N(p,r,\epsilon) \mbox{ is not smooth.}\}
\end{split}
\end{equation} 
\noindent
Note that $\mbox{sing}(\Sigma_\infty)\subset \mathcal{S}(I)$ but $\mathcal{S}(I)$ may be a strictly larger set (and in our setting this will turn out to be the case).

First we show

\begin{prop}\label{limitclass}
For some subsequence $J\subset I$, the set $\mathcal{S}(J)\setminus\{\mbox{z-axis}\}$ is one of the following:  
\begin{enumerate} 
\item empty 
\item a circle $C$ centered about the origin in $H$. 
\end{enumerate}
In the latter case,  the density of $\Sigma_\infty$ at $C$ is $2$.
\end{prop}

\begin{proof}

The $F$-stationary integral varifold $\Sigma_\infty$ is $\mathbb{P}_{\infty}$-invariant. Moreover, because 
\begin{equation}
\limsup_{g\to\infty} \lambda(\Sigma_g)\leq 2, 
\end{equation}
it follows that \begin{equation}\lambda(\Sigma_\infty)\leq 2.\end{equation}  Indeed, otherwise, there exists $(x,t)\in\mathbb{R}^2\times (0,\infty)$ so that $F_{x,t}(\Sigma_\infty)$ is arbitrarily close to a number greater than $2$.  But by the continuity of each $F$-functional under varifold convergence, this implies $F_{x,y}(\Sigma_g)$ is also greater than $2$ for a subsequence of $g$, contradicting the fact that $\lambda(\Sigma_g)<2$. 

By the entropy bound $\lambda(\Sigma_\infty)\leq 2$ and rotational symmetry, the density $\Theta(\Sigma_\infty,x)$ at any singular point $x\in\mbox{supp}(\Sigma_\infty)$ is $2$ or $3/2$.   As a limit of orientable surfaces, it follows that $3/2$ does not occur. Any tangent cone to $\Sigma_\infty$ at a point in $\mbox{sing}(\Sigma_\infty)$ consists of two distinct planes, and by Allard's theorem \cite{All} a tangent cone to $\Sigma_\infty$ at a point in $\mathcal{S}(I)\setminus\mbox{sing}(\Sigma_\infty)$ consists of a single plane with multiplicity $2$. 

Choose a subsequence $J\subset I$ and corresponding point $q\in\mathcal{S}(J)\setminus \{\mbox{z-axis}\} \subset \mathcal{S}(I)\setminus \{\mbox{z-axis}\}$ (with $p,r$ chosen so that $q\in C(p,r)$) where for any sufficiently small $\epsilon$ the convergence $\Sigma_{_J}\mres N(p,r,\epsilon)\to \Sigma_\infty\mres N(p,r,\epsilon)$ is not smooth (and there is no subsequence for which it is smooth). 

Let us partition $\mathcal{S}(J)=\mathcal{S}_0(J)\cup\mathcal{S}_1(J)$, where we set \begin{equation}\mathcal{S}_0(J)=\mathcal{S}(J)\cap H\mbox{  and   } \mathcal{S}_1=\mathcal{S}(J)\setminus\mathcal{S}_0(J) \end{equation} 


Let us denote by $N(\epsilon)$ the set $N(p,r,\epsilon)$.  We claim that for the subsequence $J$ of $g$:
\begin{equation}\label{isg}
genus(\Sigma_g\cap N(\epsilon_g)) = g \mbox{ for } g \mbox{ large enough}.
\end{equation}
for some $\epsilon_g\in [\epsilon/2,2\epsilon]$ for which $\Sigma_g$ intersects $N(\epsilon_g)$ transversally.  But if $q\in \mathcal{S}_1(J)$ this is a contradiction as by the reflective symmetry the genus of $\Sigma_g$ near $C(-p, r)$ is also $g$.  Thus $\mathcal{S}_1(J)\setminus\{\mbox{z-axis}\}=\emptyset$.  
Applying the identical argument to $\mathcal{S}_0(J)$ implies that $\mathcal{S}_0(J)\setminus\{\mbox{z-axis}\}$ contains at most one point.   Thus it suffices to prove the claim.

Consider \begin{equation}\Gamma'_g:=\Sigma_g\cap N(\epsilon_g)\end{equation} and \begin{equation}\Gamma_g=\Gamma'_g/\mathbb{Z}_{g+1},\end{equation}  where $\mathbb{Z}_{g+1}$ is the subgroup of $\mathbb{P}_{g+1}$ corresponding to rotations by angle $2\pi/(g+1)$ about the $z$-axis.  
Since $N(\epsilon_g)$ is disjoint from the $z$-axis,  this action is free and we get
\begin{equation}\label{gb}
\chi(\Gamma_g')=(g+1)\chi(\Gamma_g).  
\end{equation}
Denote by $h$ the genus of $\Gamma_g$, and $h'$ the genus of $\Gamma_g'$.  Let $e_L(g)$ denote the number of connected curves of $\partial\Gamma_g$ that lift to connected curves in $\partial\Gamma'_g$.  Recall that the first integral homology group of the torus  $\partial N(\epsilon)$ is equal to $\mathbb{Z}\oplus\mathbb{Z}$, where the first factor denotes the number of meridianal loops and the second longitudinal loops.  The only elements corresponding to embedded curves are of the form $(k,l)$ where $k$ and $l$ are relatively prime. The number $e_L(g)$ counts the curves of type $(k,l)$ for $l\neq 0$.

By untangling \eqref{gb} using the fact that boundary curves of $\Gamma_g$ that do not lift to connected curves must lift to $g+1$ such curves in $\Gamma_g'$ we obtain
\begin{equation}
h' = \frac{ge_L(g)}{2}-g+h(g+1), 
\end{equation}
We also have by assumption \begin{equation}\label{isless}h'\leq g.\end{equation}
Let $C'_g$ denote those components of $\Gamma_g'$ that contribute to the limit $\Sigma_\infty$ in $N(\epsilon/2)$.  By the monotonicity formula, $C_g'$ consists of at most $\Lambda$ components (independent of $g$).

If any element of $C_g'$ has a boundary curve $\partial\Gamma_g'$ of type $(k,l)$ for $k\geq 1$ and a sequence $g\to\infty$, then we get that $\partial N(\epsilon')$ for some $\epsilon'\in [\epsilon/2,\epsilon]$ is contained in the the support of $\Sigma_\infty$, which is impossible. 

Thus for large $g$, each component in $C'_g$ has all boundary curves of the form $(0,1)$ or $(0,0)$. Suppose such a component satisfies $e_L(g)=0$ in which case all its boundary curves are of the form $(0,0)$.  Then the only possible solution to \eqref{isless} and \eqref{gb} is if we set $h=1$ and $h'=1$.  Thus components with $e_L(g)=0$ have genus $1$.  If instead $e_L(g)=2$, we must have $h=0$ and $h'=0$.  Thus such components give a union of planar domains.

Let $C_g''\subset \Sigma_g'$ denote those connected components with $e_L(g)\leq 2$ and thus genus $0$ or $1$.  By the bound on the cardinality of $C_g'$, the combined genus of $C_g'$ is at most $\Lambda$.   By Ilmanen's integrated Gauss-Bonnet argument, (\cite{I}, \cite{K}) we obtain
\begin{equation}\label{second}
\int_{C_g''\cap N(\epsilon/2)} |A|^2 d\mu \leq \Lambda'< \infty.  
\end{equation}

By standard $\epsilon$-regularity theory \cite{CS}, the bound \eqref{second} ensures that the convergence of $C_g'$ in $N(\epsilon/2)$ is smooth away from finitely many points.  However, as $C_g'$ is invariant under $\mathbb{P}_{g+1}$ for larger and larger $g$, this is impossible unless the set of finitely many points is in fact empty.  Thus the surfaces $C_g'$ converge to some integral varifold limit $\Sigma'_\infty$ smoothly in $N(\epsilon/2)$ with density $1$ or $2$ at $(r_j,0,p_j)$.  From the definition of $\mathcal{S}$, we get that $C_g'\setminus C_g''$ must converge to its limit non-smoothly in $N(\epsilon/2)$. This forces the density of $\Sigma_\infty$ at $(r_j,0,p_j)$ to be at least $3$, contradicting the density bound of $2$.

Thus we have shown that all components in $C'_g$ must have $e_L(g)\geq 4$ for $g$ large enough.  If a component satisfies $e_L = 4$, the only solutions to \eqref{isless} and \eqref{gb} occur when $h'=g$ and $h=0$.  If a component satisfies $e_L(g)\geq 6$, there are no solutions to \eqref{isless} and \eqref{gb}.  Thus $C'_g$ contains exactly one element of genus $g$ and the genus of $\Sigma_g\cap N(\epsilon/2)$ for large $g$ is also equal to $g$, as desired.

This completes the proof of the claim, and thus the proposition.

\end{proof}

Now we show that $\mbox{sing}(\Sigma_\infty)=\emptyset$.  The argument amounts to considering possible behavior of $F$-stationary integral $1$-varifolds with at worst a single point of density $2$ (i.e. a crossing singularity):
\begin{lemma}
The set $\mbox{sing}(\Sigma_\infty)$ is empty. Moreover, $\Sigma_\infty = 2H$.   
\end{lemma}

\begin{proof}

Let $\mathcal{F}$ denote the closure of a fundamental domain of the $\mathbb{P}_{\infty}$-action given by the solid quadrant of the plane $P_1$ (containing the $z$-axis and ray $R_1$) with non-negative $x$ and $z$-coordinates.    Let us denote the non-negative $x$-axis by $X^+$ and the non-negative $z$-axis by $Z^+$.   So \begin{equation}\partial \mathcal{F} = Z^+\cup X^+\mbox{ with } X^+\cap Z^+= (0,0,0).\end{equation} 

Let $\pi:\mathbb{R}^3\rightarrow\mathcal{F}$ denote the projection to the fundamental domain.  Assume toward a contradiction that $\mbox{sing}(\Sigma_\infty)$ is non-empty.  Fix $S\in\mbox{sing}(\Sigma_\infty)$ and let $s=\pi(S)$.  Set $\tilde{\Sigma}_\infty:=\pi(\Sigma_\infty)$.  By Proposition \ref{limitclass},  $s\in Z^+\cup X^+$.


First observe that $s\notin Z^+\setminus (0,0,0)$.  Indeed, if not, then locally near such a point, the stationarity of $\Sigma_\infty$ and the fact that $\mathcal{S}_1(J)\setminus\{\mbox{z-axis}\}$ is empty by Proposition \ref{limitclass} implies that $\tilde{\Sigma}_\infty$ near $s$ consists of two otherwise disjoint segments meeting the axis $Z^+$ at angle $\pi/2$ at $s$.  Thus by the maximum principle, the two segments coincide and $s$ would in fact not be in $\mbox{sing}(\Sigma_\infty)$.  Similar reasoning implies $s\neq (0,0,0)$.  Thus we have shown $s\in X^+\setminus (0,0,0)$ and by Proposition \ref{limitclass}, it is the only $s\in\mathcal{F}$ with $\pi^{-1}(s)\subset \mbox{sing}(\Sigma_\infty)$.



It follows that $\Sigma_\infty$ has support a union of (potentially immersed) self-shrinkers.  If there are two or more self-shrinkers in the union, since each has entropy at least $1$, the entropy bound $\lambda(\Sigma_\infty)\leq 2$ implies that $\Sigma_\infty$ is a union of two planes, which is impossible unless the planes coincide with $H$. Thus $\Sigma_\infty$ is the varifold given by a single immersed self-shrinker with multiplicity $1$.  Since $\Sigma_\infty$ has an immersed circle $S$, by considering the $F_{x_0,t_0}$-functionals for $t_0\to 0$ and $x_0\in S$ we get $\lambda(\Sigma_\infty) = 2$.  However, the supremum of the $F_{x_0,t_0}$ functionals on a self-shrinker (that does not split off a line) is attained uniquely at $(1,0)$ and not in the limit that $t_0\to 0$ (Lemma 7.10 in \cite{CM}).  This contradiction establishes that $\mbox{sing}(\Sigma_\infty)$ is empty.

Again by the classification of rotationally symmetric shrinkers (Proposition \ref{classification}) and entropy bound we obtain that $\Sigma_\infty$ is either equal to $2H$, $\mathbb{S}^2_*$, $\mathbb{S}^1_*\times\mathbb{R}$, or the Angenent torus.  Note
\begin{equation}
\lambda(\mathbb{S}_*^1\times\mathbb{R})+\varepsilon_{BW}\leq \lim_{g\to\infty}\lambda(\Sigma_g)=\lambda(\Sigma_\infty)\leq 2.
\end{equation}

Thus all options for $\Sigma_\infty$ are excluded except for the Angenent torus and $2H$. If $\Sigma_\infty$ is the Angenent torus, then $\Sigma_g$ would be diffeomorphic to a torus by Allard's theorem \cite{All} for large $g$,  which is not possible by Proposition \ref{surgeries}.  This completes the proof.
 \end{proof}

Finally let us show
\begin{prop}\label{decomp}
For the subsequence $I$ and further subsequence $J\subset I$ there exists a circle $C(0,r)\subset H$ so that the convergence\begin{equation}\Sigma_{J}\rightarrow 2H\mbox{  
 as varifolds}\end{equation} is smooth and graphical on compact subsets of $H\setminus C(0,r)$.  
Moreover, for any $R>r$ and $g$ far enough along in the subsequence $J$ the set \begin{equation}\Sigma_g\cap H\cap B_R(0)\end{equation} is a union of $g+1$ embedded circles.  
\end{prop}
\begin{proof}
Suppose $\mathcal{S}(I)$ is empty.  In this case,  for any $\epsilon_i >0$ and $ R_i>0$ denote $\Omega_i =  B_{R_i}(0)\setminus B_{\epsilon_i}(0)$.  Then $\Sigma_g\to 2H$ smoothly on $H\cap\Omega_i$ for any subsequence.  Consider a sequence $R_i\to \infty$ and $ \epsilon_i\to 0$.  Then for each $i$ we get a subsequence of $g$ so that $\Sigma_g\cap\Omega_i$ may be written as a union of two normal graphs $w^1_g(x)$ and $w^2_g(x)=-w^1_g(x)$.  Set \begin{equation}u_{g}(x) = w^1_g(x)/w^1_g(p)\end{equation} for some choice of $p$ in all $\Omega_i$.  Passing to a diagonal subsequence (cf. Appendix in \cite{CM}) we obtain a non-negative rotationally symmetric solution $\phi:H\setminus (0,0)\rightarrow\mathbb{R}$ to the Jacobi equation \begin{equation}\label{jacobi}L_H \phi = 0.\end{equation}  By Proposition \ref{nofunction},  there is no such function \footnote{Alternatively, one could argue that any singularity of such a Jacobi function has to be removable \cite{CM},  and any smooth positive solution to \eqref{jacobi} on $H$ would force $H$ to be stable,  which it is not \cite{CM2}.}. Thus $\mathcal{S}(I)$ is not empty and by Proposition \ref{limitclass} we may pass to a further subsequence $J$ for which $\mathcal{S}(J)$ contains exactly one point, as desired.

Finally choose $R$ greater than the radius of the single circle contained in $\mathcal{S}(J)$. We will show that for $g$ large, $\Sigma_g\cap H\cap B_R(0)$ is a union of $g+1$ circles. This follows from symmetry considerations together with the structure obtained in Proposition \ref{limitclass}.  Indeed, by Proposition \ref{limitclass} and the smooth convergence of $\Sigma_g$ to $2H$ away from a single circle $C$, we obtain (letting $N(\epsilon)$ denote the $\epsilon$-tubular neighborhood about $C$) that 
\begin{enumerate}
    \item $\Sigma_g\cap N(\epsilon)$ is a connected surface with genus zero in each fundamental domain of the $\mathbb{Z}_{g+1}\subset\mathbb{P}_{g+1}$ action on $\mathbb{R}^3$ 
    \item $\Sigma_g\cap \partial N(\epsilon)$ consists of four longitudinal circles, two of which are contained in $\{z>0\}$ and two of which are contained in $\{z<0\}$. 
\end{enumerate}  

Let us denote \begin{equation}\Sigma_g^+=\Sigma_g\cap \{z\geq 0\}\cap B_R(0)\end{equation} and $\Sigma_g^-=\tau_H(\Sigma_g^+)$.  Consider the the quotients $\Gamma_g=(\Sigma_g\cap N(\epsilon))/\mathbb{Z}_{g+1}$ together with $\Gamma_g^\pm = \Sigma_g^\pm /\mathbb{Z}_{g+1}$.  Then $\Gamma_g$ has $4$ boundary curves and genus $0$ and thus Euler characteristic $-2$.  

Let $e$ denote the number of circles of intersection of $\Gamma_g^\pm$ with the plane $H$.  Then by the additive properties of the Euler characteristic we obtain
\begin{equation}\label{eulerdouble}
\chi(\Gamma_g) = 2\chi(\Gamma_g^+) = 2(2-(2+e)) = -2e.
\end{equation}

Since $\chi(\Gamma_g) = -2$, we get from \eqref{eulerdouble} that $e=1$.  Note that $e$ cannot be a longitudinal curve as the number of such curves is even.  Thus the single curve in $\Gamma_g\cap H$ lifts to $g+1$ distinct simple curves in $\Sigma_g\cap H\cap B_R(0)$.  Note finally that by the smooth convergence of $\Sigma_g$ to $2H$ outside of $N(\epsilon)$, these are the only curves contained in $\Sigma_g\cap H\cap B_R(0)$ for large $g$.   This completes the proof. 
\end{proof}

Note that we also have established the following useful fact:

\begin{prop}\label{bigL}
There exists $L>0$ so that for all $g$ sufficiently large,  then $\Sigma_g\cap B_L(0)$ has genus $g$.
\end{prop}
\begin{proof}
Suppose not.  Then by Proposition \ref{surgeries}, for each $L_k>0$ we get a subsequence of $g$ so that $B_{L_k}(0)\cap\Sigma_g$ has genus $0$.  Thus by Ilmanen's integrated Gauss-Bonnet argument and the increasing equivariance (as in the proof of Proposition \ref{limitclass}, we get that for this subsequence $\Sigma_g\to 2H$ smoothly in $B_{L_k/2}\setminus (0,0,0)$.  Taking $L_k\to\infty$ and a diagonal argument gives a smooth positive Jacobi field on $H\setminus (0,0)$, which is impossible by Proposition \ref{nofunction}.
\end{proof}




\section{Ends of $\Sigma_g$}\label{endssection}
In this section we prove Theorem \ref{main}\ref{5} which we restate
\begin{prop}
For $g$ large,  the self-shrinker $\Sigma_g$ has two graphical asymptotically conical ends $E_g^+$ and $E^-=\tau_H(E_g^+)$ with $E_g^+\subset\{z>0\}$.  The links $\mathcal{L}(E_g^{\pm})$ each converge to the equator $\mathbb{S}^2\cap\{z=0\}$ in the $\mathcal{C}^0$-topology.
\end{prop}

Recall that the \emph{asymptotic cone} $\mathcal{C}(\Gamma)$ of a self-shrinker $\Gamma\subset\mathbb{R}^3$ is defined to be
\begin{equation}
\mathcal{C}(\Gamma) = \lim_{\tau \to 0^+} \tau\Gamma, 
\end{equation}
and its \emph{link} is given by 
\begin{equation}
\mathcal{L}(\Gamma) = \mathcal{C}(\Gamma)\cap\mathbb{S}^2.
\end{equation}

A circle contained in $\mathcal{L}(\Gamma)$ corresponds to a conical end,  and a point corresponds to a cylindrical end.   

L. Wang \cite{LW} (extended by Sun-Wang \cite{SW}) has shown that for any self-shrinker of finite genus, $\mathcal{L}(\Gamma)$ consists of finitely many disjoint simple closed curves and points.    For $R$ large and transverse to $\Sigma_g$, the surface $\Sigma\setminus B_R(0)$ thus consists of self-shrinking (conical or cylindrical) ends $\Gamma_1,...,\Gamma_k$,  where each $\Gamma_i$ is diffeomorphic to $\mathbb{S}^1\times[0,\infty)$.

Using the decomposition provided by Proposition \ref{decomp} together with topological considerations based on the manner in which the sweepout surfaces are set up we first show that the ends of $\Sigma_g$ are disjoint from the plane $H$:
\begin{prop}\label{shown}
There exists $R>0$ so that for $g$ large enough,  
\begin{equation}
\Sigma_{g}\setminus B_{R}(0) \mbox { consists of two genus zero components, } E_g^+\mbox{ and }  E_g^-,  
\end{equation}
where $\tau_H(E_g^+)  = E_g^-$ and $E_g^+\subset\{z>0\}$.   Each 
$\partial E_g^+$ and $\partial E_g^-$ consists of one circle in $\partial B_R(0)$.
\end{prop}

\begin{proof}
Observe that for $R>L$ (from Proposition \ref{decomp}) we get that \begin{equation}genus(\Sigma_g\cap B_R(0))=g\end{equation} and from the graphical convergence of $\Sigma_g$ to $2H$,  it follows that for $g$ large enough  the set $\partial (\Sigma_g\cap B_R(0))$ consists of two simple closed curves $A^+_g\subset\{z>0\}$ and $A^-_g=\tau_H(A^+_g)\subset\{z<0\}$.

We show that no component $E$ of $\Sigma_g\setminus B_R$ contains a point of $H$.  To that end, first fix $g$ so that the statements of the previous paragraph hold.  Then consider the min-max sequence $\{\Gamma_i\}_{i=1}^\infty$ converging in the sense of varifolds to $\Sigma_g$.  Thus $\Gamma_i\cap B_R(0)$ has genus $g$.  Let $P>R$ be so large so that by Proposition \ref{decomp} $\Sigma_g\cap \partial B_P(0)$ consists of those circles $C_1,...C_k$ which bound ends diffeomorphic to $\mathbb{S}^1\times[0,\infty)$ on $\Sigma_g$.  By reflective symmetry, we have $\tau_H(C_i)= C_i$ for any $i$ with $C_i\cap H\neq\emptyset$.  Recall from Proposition \ref{genustheorem} that we may perform finitely many neckpinch surgeries on $\Gamma_i$ in $B_{P+1}(0)\cap T_\epsilon(\Sigma_j)$ together with isotopies to arrive at a surface $\Gamma_i'$ isotopic (and parallel) to $\Sigma_g$ in $B_P(0)\cap T_\epsilon(\Sigma_g)$. 

We then perform neckpinch surgeries along the subcollection \begin{equation}\{C_{i_1},...,C_{i_l}\}\subset \Gamma_i'\cap B_P(0)\end{equation} which intersect $H$ by adding in the disks in $\partial B_P(0)$ orthogonal to $H$.  This gives a closed connected surface $\Gamma_i''$ with some boundary circles $W_1,...,W_v$ in $H$.  Moreover, $genus(\Gamma_i''\cap B_R(0)) = g$ and thus 
\begin{equation}\label{genusaway} genus(\Gamma_i''\setminus B_R(0)) = 0.\end{equation}

In fact, $v=1$.  Otherwise, the surface $(\Gamma_i''\setminus B_R(0))\cap \{z\geq 0\}$ would be a planar domain with at least two boundary circles in $H$ (and one at $\partial B_R(0)$), which implies in light of \eqref{genusaway}:
\begin{equation}
\chi(\Gamma_i''\setminus B_R(0))= 2\chi(\Gamma_i''\setminus B_R(0))\cap \{z\geq 0\})<0.
\end{equation}
But $\Gamma_i''\setminus B_R(0)$ has genus $0$ by \eqref{genusaway} and two boundary components and thus
\begin{equation}\chi(\Gamma_i''\setminus B_R(0))=0.\end{equation}
This gives a contradiction and thus $v=1$.  Since $v=1$, by $\mathbb{P}_{g+1}$-equivariance, the curve $W_1$ is preserved by the group $\mathbb{Z}_{g+1}$ of rotations about the $z$-axis.  Therefore the curve $W_1$ intersects the rays $R_1$ and $R_2$ (at least) once and the surface $\Gamma_i''$ has intersection type with $b=1$ and $k\geq 4$ in \eqref{rh}, implying that it has genus at least $g+1$. This is a contradiction as $\Gamma_i''$ is obtained from $\Gamma_i$ after surgeries which can only decrease the genus.  
\end{proof}

In light of Proposition \ref{shown}, each asymptotic cone $\mathcal{C}(\Sigma_g)$ may be partitioned 
\begin{equation}
\mathcal{C}(\Sigma_g) = \mathcal{C}^+(\Sigma_g)\cup \mathcal{C}^{-}(\Sigma_g), 
\end{equation}
where $\mathcal{C}^+(\Sigma_g) = \mathcal{C}_g^1\cup ...\cup \mathcal{C}_g^{N_g}$ and each $\mathcal{C}_g^i$ is regular a cone (or ray) all of whose support aside from the tip are contained in $\{z>0\}$ and where
$\mathcal{C}^-(\Sigma_g)= \mathcal{D}_g^1\cup ...\cup\mathcal{D}_g^{N_g}$ with $\mathcal{D}_g^i = \tau_H(\mathcal{C}_g^i)$ for each $i=1,...,N_g$.  

We have the following (using the argument of Lemma 4.1 in \cite{SW}): 

\begin{lemma}[Asymptotic Cones]\label{hausdorff}
For large $g$, the set $\mathcal{C}^+(\Sigma_g)$ is non-empty and in particular $\Sigma_g$ is non-compact.  The link of $\mathcal{C}^+(\Sigma_g)$ converges in the Hausdorff topology to a subset of $\mathbb{S}^2\cap\{z=0\}$ as $g\to\infty$.  
\end{lemma}
\begin{proof}

For each $\epsilon\geq 0$ let us denote the conical tubular neighborhood (using cylindrical coordinates for $\mathbb{R}^3$):
\begin{equation}
\mathcal{U}_\epsilon =\{(r, \theta, z)\in\mathbb{R}^3\;|\; |z|\leq \epsilon r\}.  
\end{equation}
Fix $\epsilon>0$.  We claim for any $R>L$ and $g$ large enough (depending on $\epsilon$ and $R$): 
\begin{equation}\label{contra}
\Sigma_g\setminus B_{R}(0) \subset \mathcal{U}_\epsilon.   
\end{equation}
Suppose not.  Then there are a subsequence of $g$ (not relabelled) with $\rho_g y_g\in \partial\mathcal{U}_\epsilon$ where $\rho_g\to\infty$, $|y_g|=1$ and $y_g$ is contained in the northern hemisphere of $\mathbb{S}^2$. We can assume that up to a subsequence $y_g\to y$, where $y$ is contained in $\{z>0\}$.  Since $\Sigma_g\rightarrow 2H$ on compact subsets,  it follows from Proposition \ref{shown} that $\rho_g\to\infty$.   Because $\Sigma_g$ is smooth, there exists $s_g>0$ small enough so that 
\begin{equation}\label{god}
F_{\rho_g y_g, s_g}(\Sigma_g)>1-1/g.
\end{equation}
Choose $a_g=(s_g-1)/\rho_g^2$  so that by rewriting \eqref{god} we obtain for each $g$ 
\begin{equation}\label{hm}
F_{\rho_g y_g, 1+a_g\rho_g^2}(\Sigma_g)>1-1/g.
\end{equation}
\noindent
Note that we may assume $a_g\to 0$.

Recall that for $\Gamma$ a self-shrinker,  we have (equation 7.13 in \cite{CM}) for $y\in \mathbb{R}^3$ and $a\in\mathbb{R}$ and $s>0$
\begin{equation}
F_{sy, 1+as^2}(\Gamma)\mbox{ is non-increasing in } s.  
\end{equation}

Applying this monotonicity formula to \eqref{hm} we get for any $\rho>0$
\begin{equation}\label{lower}
F_{\rho y, 1}(2H) = \lim_{g\to\infty} F_{\rho y_g, 1+a_g\rho^2}(\Sigma_g) \geq \liminf_{g\to\infty} F_{\rho_g y_g, 1+a_g\rho_g^2}(\Sigma_g)\geq 1.
\end{equation}
The equality in \eqref{lower} is from the continuity of the $F$-functionals with respect to varifold convergence (cf. Proposition 2.4 in \cite{SW}) together with the fact that $a_g\to 0$.  On the other hand,  as $\rho\to\infty$, we have \begin{equation} F_{\rho y, 1}(2H)\to 0,\end{equation} since $y$ is not in the support of $H$.

Since \eqref{contra} holds for genera $g$ sufficiently large depending on $\epsilon>0$, and $\Sigma_g\setminus B_R(0)$ is disjoint from $H$ by Proposition \ref{shown}, the conclusion follows.   
\end{proof}

We now show there are two distinguished ``large" ends to $\Sigma_g$:
\begin{lemma}\label{homologous}
Fix $\epsilon<1/2$.  Then for each $g$ large enough,  there is exactly one index $i\in\{1,...,N_g\}$ with the property that
\begin{equation}\label{arenontrivial}
\mathcal{C}_g^i\cap\mathbb{S}^2 \mbox{ is homotopic in }  \mathcal{U}_\epsilon\cap\mathbb{S}^2 \mbox{ to the equator }\{z=0\}\cap\mathbb{S}^2.  
\end{equation}
\end{lemma}
\begin{proof}
For $L_g>L$ chosen large enough, it follows from Lemma \ref{hausdorff} that $\Sigma_g\cap (\partial B_{L_g}(0)\cap\mathcal{U}_\epsilon)$ consists of $N_g$ curves. Suppose we order the curves so that the first $N_g'\leq N_g$ are longitudinal curves in the solid (piecewise smooth) torus $T=\mathcal{U}_\epsilon\cap (B_{L_g}(0)\setminus B_L(0))\cap\{z>0\}$.  Note that $N_g'+1$ is even since $\partial(\Sigma_g\cap T)$ is null-homologous in $T$ and since $\Sigma_g\cap \partial B_L(0)$ consists of one curve by Proposition \ref{shown}.  Thus $N_g'$ is odd and in particular at least $1$.  Since $L_g$ may taken arbitrarily large without changing $N_g'$, this gives \eqref{arenontrivial} for at least one suitable index $i$.  
To see that $N_g'=1$, observe that since entropy is non-increasing under MCF:
\begin{equation}\label{up}
\sum_{i=1}^{N_g} (F_{0,1}(\mathcal{C}^i_g)+F_{0,1}(\mathcal{D}^i_g))\leq\lambda(\Sigma_g)\leq 2, 
\end{equation}
and thus by symmetry
\begin{equation}\label{upp}
\sum_{i=1}^{N_g} F_{0,1}(\mathcal{C}^i_g)\leq 1.
\end{equation}
On the other hand,  in light of \eqref{arenontrivial} and Lemma \ref{hausdorff},  for $i=1,...,N_g'$,  and any $\delta>0$ we obtain that \begin{equation}F_{0,1}(\mathcal{C}^i_g)> 1-\delta\end{equation} for $g$ large enough.  If $N_g'>1$ this contradicts \eqref{upp}.  
\end{proof}

Our goal us to show that $E_g^\pm$ is each diffeomorphic to an infinite annulus $\mathbb{S}^1\times[0,\infty)$.   Equivalently,  our goal is to show $N_g=1$ for each $g$ large enough.   The danger is that $E_g^\pm$ has ends sprouting up at larger and larger radii on $\Sigma_g$.

We need Ecker's localized version of Huisken's monotonicity formula \cite{E2}.  Namely,  suppose $\mathcal{M}= M_t$ is an integral Brakke flow and set $X_0=(x_0,t_0)$.   For 

\begin{equation}\phi_{X_0}(x,t) = (1-\frac{|x-x_0|^2+2n(t-t_0)}{\rho^2})^3_+\end{equation} 

and

\begin{equation}
\rho_{X_0}(x,y) = (4\pi(t_0-t))^{-1/2} e^{-\frac{|x-x_0|^2}{4(t_0-t)}}
\end{equation}

The Gaussian density ratio 
\begin{equation}
\Theta(\mathcal{M}, X_0,r) = \int_{M_{t_0-r^2}} \rho_{X_0}\phi^\rho_{X_0}d\mu
\end{equation}
is non-increasing in $r$.  

The density at $X_0$ is defined as
\begin{equation}
\Theta(\mathcal{M}, X_0)=\lim_{r\to 0} \Theta(\mathcal{M}, X_0,r).
\end{equation}

The density $\Theta(\mathcal{M}, X_0)$ is equal to $1$ when $X_0$ is a smooth point of the flow. We need one version of White's version of Brakke regularity theorem for limits of Brakke flows:

\begin{thm}[Brakke regularity theorem \cite{W}]\label{brakke}
Suppose $\mathcal{M}_i$ are smooth Brakke flows with $\mathcal{M}_i\rightharpoonup\mathcal{M}$.  If \begin{equation}
\Theta(\mathcal{M},X_0) = 1, 
\end{equation}
then there is an open neighborhood of $X_0$ in space-time so that in this neighborhood, the convergence of $\mathcal{M}_i$ to $\mathcal{M}$ is smooth.
\end{thm}

With this background, we can now show
\begin{prop}
For $g$ large enough,  $E_g^\pm$ are annuli that may be expressed as normal graphs over the plane $H$.
\end{prop}
\begin{proof}
Assume toward a contradiction for some subsequence (not relabelled) $g\to\infty$ there are points $p_g\in E_g^+$ so that the tangent plane to $\Sigma_g$ is orthogonal to $H$ at $p_g$ and $R_g=|p_g|\to\infty$.    


Pass to a subsequence of $g$ so that $q$ is the limit of $R_g^{-1}p_g$ with $q\in\mathbb{S}^2\cap H$.  Thus there exists $\epsilon_g\to 0$ so that
\begin{equation}\label{meaninglimit}
|q-R_g^{-1}p_g|\leq\epsilon_g.
\end{equation}

Let us consider the flows \begin{equation}\label{inball}\overline{\Sigma}_g(t) := \sqrt{-t}\Sigma_g\mres B_{1/2}(q)\end{equation} for $t\in [-1,0]$.  

For $t\in [-t_0,0]$ for some $t_0$ small enough, by Proposition \ref{shown} for all $g$ large enough, the flow \eqref{inball} is disjoint from the plane $H$.    Let $\hat{\Sigma}_g (t)$ denote the connected components of $\overline{\Sigma}_g(t)$ with positive $z$-coordinate and let $\{\hat{\mu}^g_t\}_{t\in [-t_0,0]}$ denote the corresponding integral Brakke flows in $B_{1/2}(q)$.  By the compactness theorem for Brakke flows,  we may extract a weak limit \begin{equation}\hat{\mu}_t^g\rightharpoonup\hat{\mu}_t.\end{equation}  From the definition of weak convergence for such flows,  $\hat{\mu}_t$ is the varifold limit of $\hat{\mu}^g_t$ for each $t\in [-t_0,0]$.

We claim $\hat{\mu}_t$ is the static flow of planes near $q$ for $t\in [-t_0,0]$. For fixed $t\in [-t_0,0)$, observe that the varifold $\hat{\mu}_t^g$ converges to $H \mres B_{1/2}(q)$ as $g\to\infty$.  Indeed by Proposition \ref{mainsec}, since $\Sigma_g\to 2H$ on compact subsets of $H$, it follows that for any fixed $\lambda>0$, $\lambda \Sigma_g\to 2\lambda H$ on any compact set as $g\to\infty$.  Setting $\lambda = \sqrt{-t}$ gives the claim noting that $H$ is invariant under homothety. 

For $t=0$,  on the other hand, Lemma \ref{hausdorff} implies that the collection of asymptotic cones $\hat{\mu}_0^g$ converges in the sense of integral currents to $H$ as $g\to\infty$. It follows that  \begin{equation}\label{lowerdensity} \Theta(\hat{\mu}_0, (q,0))\geq 1.\end{equation}   
Applying Ecker's localized monotonicity formula together with \eqref{lowerdensity} gives
\begin{equation}\label{mon}
\lim_{r\to 0} \Theta(\hat{\mu}_t, (q,0),r^2) \geq \Theta(\hat{\mu}_0,(q,0))\geq 1. 
\end{equation}
Since $\hat{\mu}_t$ consists of an open subset of a plane for $t<0$, the limit in \eqref{mon} is $1$, and thus 
\begin{equation}
\Theta(\hat{\mu}_t, (q,0))=1.
\end{equation}

It follows that $\{\hat{\mu}_t\}_{t\in [-t_0,0]}$ is the static flow of planes in $B_{1/2}(q)$.   By Theorem \ref{brakke} it follows that $\hat{\mu}_t^g$ converges smoothly to the static flow $\hat{\mu}_t$ in a backwards parabolic ball $ B_{1/4}(q)\times [-t_0/2,0]$. 

On the other hand, setting $t_g=-\frac{1}{R_g^2}$, we get by \eqref{meaninglimit} that for the previously chosen $p_g\in\Sigma_g$ and $g$ large enough
\begin{equation}\label{stillclose}
|q-\sqrt{-t_g}p_g|\leq\epsilon_g\leq\frac{1}{4}
\end{equation}

By the choice of $p_g$, the equation \eqref{stillclose} together with the fact that homotheties preserve angles imply that $\hat{\mu}^g_{t_g}$ is \emph{not} graphical in the ball $B_{1/4}(q)$ for the sequence of times $t_g\to 0$. This is a contradiction.
\end{proof}

\section{Stability operator $L_{0,1}$ on $H$}\label{appendix}
The stability operator on $H$ is given by (equation 5.1 in \cite{CM}) 
\begin{equation}
L_H\phi =\Delta_{\mathbb{R}^2}\phi-\frac{1}{2}\langle x, \nabla\phi\rangle + \frac{1}{2}\phi.
\end{equation}
In polar coordinates $(r,\theta)$ this becomes
\begin{equation}\label{inpolar}
L_H\phi  = \phi_{rr}+\frac{1}{r}\phi_r+\frac{1}{r^2}\phi_{\theta\theta} -\frac{r}{2}\phi_r + \frac{1}{2}.
\end{equation}
We will need to study solutions $\phi$ to $L_H\phi=0$.   Assuming the solution $\phi=\phi(r)$ is rotationally symmetric we get
\begin{equation}
L_H\phi  = \phi_{rr}+(\frac{1}{r}-\frac{r}{2})\phi_r + \frac{1}{2}.
\end{equation}
Let us make the change of variable $\xi = r^2/4$.   Thus
\begin{equation}
\phi_r = \frac{r}{2} \phi_\xi,
\end{equation}
and
\begin{equation}
\phi_{rr} = \frac{1}{2} \phi_{\xi}+\frac{r^2}{4} \phi_{\xi\xi} = \frac{1}{2} \phi_{\xi}+\xi \phi_{\xi\xi}.
\end{equation}
Thus we obtain that the rotationally invariant solutions to $L_H\phi=0$ satisfy
\begin{equation}\label{che}
\xi\phi_{\xi\xi} + (1-\xi)\phi_\xi + \frac{1}{2}\phi = 0.
\end{equation}
The equation \eqref{che} is a confluent hypergeometric equation studied by Kummer in 1837 (\cite{Ku}).   In general,  one solution 
to 
\begin{equation}\label{general}
\xi\phi_{\xi\xi} + (b-\xi)\phi_\xi - a\phi = 0, 
\end{equation}
is given by 
\begin{equation}
M(a,b,\xi) = 1+\frac{a}{b\cdot 1}\xi + \frac{a(a+1)}{b(b+1)\cdot 1\cdot 2}\xi^2 +..., 
\end{equation}
or 
\begin{equation}\label{series}
M(a,b,\xi) = \sum_{k=0}^\infty \frac{(a)_k}{(b)_k\cdot k!} \xi^k, 
\end{equation}
where the Pockhammer symbol $(a)_k$ is defined as
\begin{equation}
(a)_k = a(a+1)(a+2)...(a+k-1).
\end{equation}
This solution has asymptotics
\begin{equation}
M(a,b,\xi) \sim  \frac{e^\xi \xi^{a-b}}{\Gamma(a)} \mbox{ as } \xi\rightarrow\infty.
\end{equation}
\noindent
Thus one solution to \eqref{che} can be denoted \begin{equation}\phi_1(r) := M(-\frac{1}{2},1,r^2/4)\end{equation} and 
\begin{equation}
\phi_1(r) \sim -\frac{4r^{-3}e^{r^2/4}}{\sqrt{\pi}} \mbox{ as } r\rightarrow\infty.
\end{equation}
Observe that $\phi_1(0) = 1$ and from differentiating the series expansion we get that $\phi'_1(r)<0$ and $\phi''_1(r)<0$ for $r>0$.   It follows that $\phi_1$ has precisely one zero at $r=r_1$.   


A second linearly independent solution $U(a,b, \xi)$ to \eqref{general} is known as the Tricomi function with asymptotics 
\begin{equation}
U(a,b,\xi) = \xi^{-a} \mbox{ for large } \xi.
\end{equation}
\noindent
Thus we may set 
\begin{equation}\phi_2(r):=U(-\frac{1}{2},1,r^2/4)\footnote{In fact there is an explicit formula
\begin{equation}
U(-\frac{1}{2},1,\xi) = e^{\xi/2}((\xi-1)K_0(\xi/2) + \xi K_1(\xi/2)), 
\end{equation}
where $K_0$ and $K_1$ denote the modified Bessel functions of the second kind,  which decay with rate $\frac{e^{-\xi}}{\sqrt{\xi}}$ as $\xi\rightarrow\infty$.} \end{equation} 
which satisfies
\begin{equation}
\phi_2(r) \sim \frac{r}{2} \mbox{ for large } r.
\end{equation}
Moreover, as $r\to 0$, $\phi_2(r)\to-\infty$ logarithmically.

One also finds numerically that the function $\phi_2(r)$ has precisely one zero $r_2$ with $r_2<r_1$.


\begin{prop}\label{nofunction}
There exists no positive radial function $\phi=\phi(r)\in C^\infty(H\setminus (0,0))$ satisfying
\begin{equation}
L_{H}\phi = 0.
\end{equation}
\end{prop}
\begin{proof}
Since $\phi_1$  and $\phi_2$ change sign it suffices to show that for each $\lambda\neq 0\in\mathbb{R}$,  the function $\phi_\lambda(r):= 
\phi_2(r)+\lambda\phi_1(r)$ is not everywhere positive or everywhere negative.    But $\phi_\lambda(r_1) =\phi_2(r_1)>0$ and in light of the fact that $\phi_1(0)=1$ and  $\phi_2$ tends to $-\infty$ as $r\to 0$ we obtain that $\phi_\lambda(r)\rightarrow -\infty$ as $r\rightarrow 0$.   Since $\phi_\lambda$ takes on positive and negative values, this completes the proof.
\end{proof}


\printbibliography

\end{document}